\documentclass[review,11pt,a4paper]{elsarticle}
\usepackage{amsfonts}
\usepackage{amsmath}
\usepackage{setspace}
\usepackage{amsthm}
\usepackage{graphicx}
\usepackage[margin=.90in]{geometry}
\usepackage{algorithm}
\usepackage[noend]{algpseudocode}

\makeatletter
\def\ps@pprintTitle{%
 \let\@oddhead\@empty
 \let\@evenhead\@empty
 \def\@oddfoot{}%
 \let\@evenfoot\@oddfoot}
\makeatother

\usepackage{hyperref}
\usepackage{multirow}
\usepackage{booktabs}
\usepackage[nolist]{acronym}%
\usepackage{caption} 
\newtheorem{lemma}{Lemma}
\newtheorem{assump}{Assumption}
\newtheorem{theorem}{Theorem}

\algnewcommand\algorithmicinput{\textbf{Indices and Definitions:}}
\algnewcommand\Definitions{\item[\algorithmicinput]}%

\usepackage{xcolor}

\biboptions{authoryear}

\def\thesismode{0}
\begin{document}

\begin{frontmatter}

\title{A Moment and Sum-of-Squares Extension of Dual Dynamic Programming with Application to Nonlinear Energy Storage Problems}
\author[a]{Marc Hohmann\corref{mycorrespondingauthor}}
\author[b]{Joseph Warrington}
\author[b]{John Lygeros}
\address[a]{Urban Energy Systems Group, Empa, Swiss Federal Laboratories for Materials Science and Technology, \"Uberlandstrasse 129, 8600 D\"ubendorf , Switzerland}
\address[b]{Automatic Control Laboratory, ETH Zurich, Physikstrasse 3, 8092 Z\"urich, Switzerland}

\cortext[mycorrespondingauthor]{Corresponding author}

\begin{abstract}
 We present a finite-horizon optimization algorithm that extends the established concept of Dual Dynamic Programming (DDP) in two ways. First, in contrast to the linear costs, dynamics, and constraints of standard DDP, we consider problems in which all of these can be polynomial functions. Second, we allow the state trajectory to be described by probability distributions rather than point values, and return approximate value functions fitted to these. The algorithm is in part an adaptation of sum-of-squares techniques used in the approximate dynamic programming literature. It alternates between a forward simulation through the horizon, in which the moments of the state distribution are propagated through a succession of single-stage problems, and a backward recursion, in which a new polynomial function is derived for each stage using the moments of the state as fixed data. The value function approximation returned for a given stage is the point-wise maximum of all polynomials derived for that stage. This contrasts with the piecewise affine functions derived in conventional DDP. We prove key convergence properties of the new algorithm, and validate it in simulation on two case studies related to the optimal operation of energy storage devices with nonlinear characteristics. The first is a small borehole storage problem, for which multiple value function approximations can be compared. The second is a larger problem, for which conventional discretized dynamic programming is intractable.

\end{abstract}

\begin{acronym}
	\acro{ADP}{Approximate Dynamic Programming}
	\acro{CHP}{Combined Heat and Power plant}
	\acro{COP}{coefficient of performance}
	\acro{DDP}{Dual Dynamic Programming}
	\acro{DP}{Dynamic Programming}
	\acro{ESMP}{Energy Storage Management Problem}
	\acro{GMP}{Generalized Moment Problem}
	\acro{HP}{heat pump}	
    \acro{LMI}{linear matrix inequality}
	\acro{LP}{Linear Programming}
	\acro{MILP}{Mixed-Integer Linear Programming}
	\acro{MINLP}{Mixed-Integer Nonlinear Programming}
	\acro{NLP}{Nonlinear Programming}
	\acro{NP-hard}{Nondeterministic Polynomial time-hard}
	\acro{OPF}{Optimal Power Flow}
	\acro{PCM}{Phase-Change Materials}
	\acro{RES}{Renewable Energy Sources}
	\acro{SDP}{semidefinite program}
	\acro{SOS}{Sum-of-Squares}
\end{acronym}

\end{frontmatter}

%\begin{linenumbers}

\section{Introduction}\label{introduction}
\ac{DDP} \citep{Pereira1991}, also referred to as nested Benders decomposition, is a means of solving multi-stage optimization problems in which constraints on decision variables are coupled only across adjacent stages. The most common application is in a linear, stochastic setting, where it is referred to as Stochastic \acl{DDP} (SDDP). The algorithm relies on a Benders decomposition argument to generate increasingly tight lower bounds on the optimal cost-to-go at each stage. Convergence to optimality of these bounds and of forward state trajectories has been studied in \cite{philpott2008} for the linear case, and \cite{girardeau2015} for the general nonlinear case. Inexact approaches featuring suboptimal cuts and/or forward state trajectories were studied in \cite{zakeri2000} and \cite{guigues2018}, and a number of other extensions have been developed, notably for risk-averse decision making \citep{guigues2012} and multi-stage integer problems \citep{zou2018}. 

In a multi-stage setting, value functions allow single-stage decisions to be taken without explicit consideration of the remainder of the time horizon. This is relevant in many energy applications featuring storage of some kind, where short-term decisions must often be made in the presence of long-term effects driven by slower, for example seasonal, dynamics \citep[see][]{Abgottspon2015,Darivianakis2017}. A locally-tight approximation of the cost-to-go allows relatively efficient trade-offs between short- and long-term costs to be made, even when an exogenous disturbance, or modelling error, may have caused the system state to deviate somewhat from a previously computed trajectory. The value function approximations generated by (S)\ac{DDP} often have this property, and can therefore be well suited to this purpose.

However, a shortcoming common to many nested decomposition approaches, including (S)\ac{DDP}, is that they are only applicable to systems with linear dynamics, costs, and constraints, or with ``benign'' (convex) nonlinearities \citep{girardeau2015}. Many problems to which (S)\ac{DDP} could otherwise be applied feature nonconvex, in particular polynomial, relationships between variables. Examples of polynomial nonlinearities in the energy domain include hydro storage planning with head effects \citep{cerisola2012}, district heating networks \citep{Tang2014}, borehole management using heat pumps \citep{atam2016}, and alternating-current (AC) power system optimization \citep{Taylor2015a}. Although in some cases it is possible to apply a convex approximation, for example McCormick envelopes for bilinear functions \cite{cerisola2012}, this may not offer acceptable modelling accuracy.

For low-dimensional nonlinear systems, it is possible in a very broad range of cases to compute a near-optimal value function by discretizing the state and input spaces and performing the standard \ac{DP} recursion \citep{bertsekas1995dynamic}. This approach has been applied to seasonal borehole storage problems in \cite{DeRidder2011} and \cite{atam2016}, but it becomes impractical for systems with more than only a few states and inputs due to exponential memory and computation requirements. It is therefore desirable to extend the existing theory of \ac{DDP} to handle nonlinear systems, in order to take advantage of \ac{DDP}'s relative scalability.

Other \ac{ADP} \citep{powell2011} approaches address the drawbacks of discretized \ac{DP} by using relaxations of the dynamic programming principle, most commonly in an infinite-horizon setting. Recent approaches such as  \cite{Wang}, \cite{summers2012}, and \cite{Beuchat2017} propose tractable approximations to the \ac{LP} formulation of \ac{ADP} \citep{HernandezLerma1994}, in which the computation of a value function is cast as an (infinite-dimensional) \ac{LP}. The authors of \cite{Savorgnan2009}, \cite{Kamoutsi2017}, and \cite{Lasserre2007} formulate a \ac{GMP} over occupation measures, of which this \ac{LP} formulation is a dual. They derive tractable approximations of the \ac{GMP} and \ac{LP} formulation in the form of moment relaxations and \ac{SOS} programs for approximate control synthesis of polynomial systems. In these approaches, the optimal control problem is solved for a specified initial state distribution. It should also be noted that \acp{GMP} have gained interest recently in the energy domain outside of \ac{DP}, due to their ability to find global solutions of the AC optimal power flow problem \citep{Ghaddar2014,Molzahn2015}.

In this paper, we develop an approach that brings the advantages of the \ac{LP} formulation of \ac{ADP} to \ac{DDP}, in that it handles polynomial costs, dynamics, and constraints, and fits the value function to trajectories emanating from an initial state \emph{distribution}, in contrast to the single initial state used in conventional \ac{DDP}. 
%In an initial step, an approximate value function for each stage is obtained by solving a \ac{SOS} program. This value function approximation scheme is derived from the dual of a \ac{GMP} for finite-horizon discrete-time dynamical systems and it is shown to provide so-called \emph{subsolutions} to the Bellman equation. \joe{Maybe we could drop this paragraph altogether?} 
As with conventional DDP, the algorithm performs an iterative sequence of forward simulations and backward recursions. The forward simulation consists of moment problems approximating the occupation measure of candidate trajectories, while the backward recursion is composed of \ac{SOS} programs, dual to the moment problems, that generate under-approximators of the value function. The output of our proposed algorithm is a collection of functions for each stage, the point-wise maximum of which under-approximates the true value function. This yields a richer class of approximations than the Moment/SOS approaches of \cite{Lasserre2007} and \cite{Savorgnan2009} for polynomial dynamical systems, which rely on a single, high-order polynomial to increase accuracy. The methods developed in \cite{ODonoghue2011} and \cite{Beuchat2017} also generate a point-wise maximum under-approximation in an iterative fashion, but do not use the primal side over moments of the occupation measure to refine the approximate value functions.

Specifically, we make the following contributions:

\begin{itemize}
	\item We extend the well-known \ac{DDP} framework to generic polynomial dynamical systems using moment/SOS techniques. We define an algorithm, Moment \ac{DDP}, that generates increasingly tight lower bounds on each stage's value function, and corresponding moments of the state distribution at each stage. This algorithm generates value function estimates that are valid for a \emph{probability distribution} of initial states, encompassing the single initial state (or Dirac distribution) from conventional \ac{DDP} as a special case.
	\item We prove that (i) the upper and lower cost bounds generated by the algorithm converge to at least the optimal cost of a relaxation of the finite-horizon decision problem and at most the optimal cost of the original \ac{GMP}, and (ii) this relaxation becomes tight in the limit as the order of the moment relaxation increases.
	\item We describe the stochastic extension of Moment \ac{DDP}, and give conditions under which the uncertainty can be accommodated within the same framework.
	\item We demonstrate Moment \ac{DDP} numerically with a nonlinear seasonal geothermal borehole dispatch problem based on real measurement data. Furthermore, we report successful application of the algorithm to a higher-dimensional system that is computationally too demanding for conventional discretized \ac{DP}.
\end{itemize}

Section \ref{modelling} states the class of finite-horizon polynomial problems considered in our framework, and presents a finite-horizon discrete-time \ac{SOS} approach to \ac{ADP} inspired by recent optimal control literature. Section \ref{MomentSOS} describes the Moment \ac{DDP} algorithm, and Section \ref{convergence} states and proves its key convergence properties. Section \ref{numerics} presents numerical results for two nonlinear borehole systems of different state dimensions. Section \ref{conclusion} concludes and gives an outlook for future research.
%% NOTATION -------------------------------------
%% ----------------------------------------------
\subsection{Notation and preliminaries} \label{sec:prelim}
The sets $\mathbb{R}$, $\mathbb{N}$ and $\mathbb{N}^+$ denote the real numbers, non-negative and positive integers respectively. For a compact real vector space $\mathbf{S}$, let $\mathcal{M}(\mathbf{S})$ be the set of Borel measures on $\mathbf{S}$ and $\mathcal{C}(\mathbf{S}$) the set of bounded continuous functions on $\mathbf{S}$. Together they form a dual pair ($\mathcal{M}(\mathbf{S})$, $\mathcal{C}(\mathbf{S}$)) with duality brackets $\langle v,\mu\rangle=\int_{\mathbf{S}}v d\mu$ for $v\in\mathcal{C}(\mathbf{S})$. If $v$ is polynomial, we write the duality bracket as an inner product $\langle\mathbf{v},\mathbf{m}\rangle$, where the vector $\mathbf{v}$ contains the coefficients of $v$ and the vector $\mathbf{m}$ the corresponding moments of $\mu$. $\mathcal{M}(\mathbf{S})_+$ denotes the set of positive Borel measures on $\mathbf{S}$. A positive Borel measure $\varphi$ supported on $\mathbf{S}$ with $\varphi(\mathbf{S})=1$ is called a Borel probability measure. A special case of a Borel probability measure is a Dirac measure $\delta_x$ supported on a single point $x\in\mathbf{S}$. The operator $\otimes$ defines the cross product of two probability measures. The expected value with respect to a Borel probability measure $\varphi$ is defined as $\mathbf{E}_{\varphi}(x)=\int_{\mathbf{S}}xd\varphi$. For a Borel set $A$, we define $1_A(x)$ as an indicator function equal to 1 if $x\in A$ and 0 if $x\notin A$. 

Let $\mathbb{R}[x]_k$ be the ring of polynomials of degree at most $k$ in some variable $x \in \mathbb{R}^n$, and let $\rm deg(p)$ denote the degree of $p$. The notation $\Sigma_{2k}[x]$ stands for the Sum-of-Squares polynomials of degree at most $2k$ in $x$.  Polynomial $p(x)\in\Sigma_{2k}[x]$ if and only if there exist polynomials $\xi_1(x),\ldots,\xi_{N_\xi}(x)$ such that $p(x)=\sum_{i=1}^{N_\xi} \xi_i(x)^2$, which implies that $p(x)\geq 0$ for all $x$. This is equivalent to there existing a symmetric, positive semidefinite matrix $\mathbf{P}$ (we denote this $\mathbf{P} \succeq 0$) such that $p(x) \equiv \tilde{p}(x)^\top \mathbf{P} \tilde{p}(x)$. In this definition, $\tilde{p}(x) := (1, x_1, x_2, \ldots, x_1x_2, \ldots, x_n^k)$ is the vector of all possible monomials in $x$, of degree up to $k$. An optimization over the elements of $\mathbf{P}$, with the \ac{LMI} constraint that $\mathbf{P} \succeq 0$, therefore yields parameterizations of \ac{SOS} polynomials as solutions. We refer to the degree of a \ac{SOS} polynomial $p(x)$ as $2k$ since $\textrm{deg}(p)$ is always an even number.

The truncated \emph{quadratic module} of degree $k$, generated by the polynomials $h_i(x)$ of a semi-algebraic set $\mathbf{S}:=\{h_i(x)\geq 0,i=1,\ldots,N_h\}$, is defined as 
\begin{equation}\label{eq:sosdef}
\textrm{Q}_k(\mathbf{S}):=\sigma_0(x)+\sum_{i=1}^{N_h}\sigma_i(x)h_i(x),
\end{equation}
where $\sigma_0\in\Sigma_{2k}[x]$ and $\sigma_i\in\Sigma_{2k}[x]$, with the restriction that $\textrm{deg}(\sigma_i h_i)\leq 2k$. Such polynomials are guaranteed to be non-negative for all $x \in \mathbf{S}$.

%% Problem statement ----------------------------
%% ----------------------------------------------
\section{Problem statement and background}\label{modelling}
\subsection{Finite horizon problem}
We consider a finite-horizon decision problem of the form \eqref{eq:energystorage}, and the corresponding optimal value $V_0^*(x_0)$ for given $x_0$:
\begin{subequations}\label{eq:energystorage}
	\begin{align}
	V_0^*(x_0):=\min_{\{x_t\}_{t=1}^T, \{u_t\}_{t=0}^{T-1}} \quad &\sum_{t=0}^{T-1} l_t(x_t,u_t) + H(x_T)\label{eq:energystorage_a}\\
	\text{s.t.} \quad &x_{t+1}=f_t(x_t,u_t), \quad t=0,\ldots,T-1,\label{eq:energystorage_b}\\
	&g_{t,j}(x_t,u_t) \geq 0, \quad j=1,\ldots,N_{g,t}, \quad t=0,\ldots,T-1,\label{eq:energystorage_d}\\
    &g_{T,j}(x_T) \geq 0, \quad j=1,\ldots,N_{g,T}.
	\end{align}
\end{subequations}
Vector $x_t \in \mathbb{R}^{n_x}$ represents the state at stage $t$, $u_t \in \mathbb{R}^{n_u}$ is a vector of control inputs (or actions), and $t=0,\ldots,T$ is the time index over a prediction horizon of length $T\in\mathbb{N}^+$. Stage costs are defined by functions $l_t: \mathbb{R}^{n_x} \times \mathbb{R}^{n_u} \rightarrow\mathbb{R}$ and the terminal cost function is $H:\mathbb{R}^{n_x}\rightarrow \mathbb{R}$. The dynamics are modelled by the function $f_t(x_t,u_t): \mathbb{R}^{n_x}\times \mathbb{R}^{n_u} \rightarrow\mathbb{R}^{n_x}$, and the constraint functions $g_{t,j}(x_t,u_t): \mathbb{R}^{n_x}\times \mathbb{R}^{n_u} \rightarrow\mathbb{R}$ encode conservation laws and technical bounds on variables at each stage. 

For later developments, we will assume that $x_t$ includes an auxiliary state $x_{c,t}$ on the interval $[0,T]$ with update equation $x_{c,t+1}=x_{c,t}+1$, thus representing the current time step $t$ as a state.

With a minor abuse of notation, we say that constraints (\ref{eq:energystorage_d}) that are uncoupled from $u_t$ define the state space $\mathbf{X}_t:=\{x_t\in\mathbb{R}^{n_x}: g_{t,j}(x_t)\geq 0, j=1,\ldots,N_{g_x,t};x_{c,t}=t\}$. Constraints (\ref{eq:energystorage_d}) that are uncoupled from $x_t$ define the action space $\mathbf{U}_t:=\{u_t\in\mathbb{R}^{n_u}: g_{t,j}(u_t)\geq 0, j=N_{g_x,t}+1,\ldots,N_{g_u,t}\}$. The feasible set of state and control decisions at time step $t$ is defined as
\begin{equation*}
\mathbf{C}_t:=\{(x_t,u_t)\in\mathbb{R}^{n_x}\times\mathbb{R}^{n_u}: g_{t,j}(x_t,u_t)\geq 0, j=1,\ldots,N_{g,t};x_{c,t}=t\}.
\end{equation*}
For any $x_t \in \mathbf{X}_t$, the set of admissible controls is defined as $\mathbf{U}_t(x_t):=\{u_t: (x_t,u_t)\in\mathbf{C}_t\}$. Since the sets $\mathbf{C}_t$ and $\mathbf{X}_t$ contain a constraint $x_{c,t}=t$, and all problem constraints will be defined for states and inputs belonging to these time-indexed sets, we will drop the time subscripts from $x$ and $u$ to maintain clean notation, without loss of clarity. We will also refer to $x_{c,t}$ as $x_c$ under the same rationale.

Furthermore, we make the following assumptions:
\begin{assump}\label{as:polynomialcompact}
	Functions $l_t(x,u)$, $f_t(x,u)$, $g_{t,j}(x,u)$ are polynomials for all $t \in \{0,\ldots,T-1\}$, as is $H(x)$. The state and control decisions are bounded i.e., $\mathbf{C}_t$ and $\mathbf{X}_t$ are compact.
\end{assump}
\begin{assump}\label{as:invariance}
	For all $t=0,\ldots, T-1$, for all $x \in \mathbf{X}_t$ there exists at least one $u \in \mathbf{U}_t(x)$ such that $f_t(x,u) \in \mathbf{X}_{t+1}$.
\end{assump}

The \emph{value function} $V_t^*:\mathbf{X}_t\rightarrow \mathbb{R}$ represents the sum of all costs incurred in problem \eqref{eq:energystorage} starting from state $x_t$ at time instance $t$, if optimal control decisions are taken at all times from $t$ to $T-1$. It is defined recursively by the well-known Bellman optimality condition at each stage $t=0,\ldots,T-1$:
\begin{equation}\label{eq:bellman}
V_t^*(x):=\min_{u\in\mathbf{U}_t(x)} \left\{ l_t(x,u)+V_{t+1}^*(f_t(x,u))\right\},\quad \forall x\in\mathbf{X}_t,
\end{equation}
with the boundary condition $V_T^*(x) = H(x)$ for all $x\in\mathbf{X}_T$.

\subsection{Generalized moment problem}\label{gmp}
We now develop a finite-horizon discrete-time optimal control problem in the form of a \ac{GMP} \citep{Lasserre2014}. Our formulation, an infinite-dimensional linear program over occupation measures, is a finite-horizon problem related to the \ac{GMP} developed in \cite{Savorgnan2009}. An occupation measure can be interpreted as a probability distribution describing the trajectory $x$ and $u$ of a dynamical system starting from a known initial state distribution. 

%\joe{Remove most of the following.}
%To maintain compact notation for the \ac{GMP}, we assume the problem data $g_{t,j}(x,u)$ and $l_t(x,u)$ to be time-invariant over the horizon $T$, i.e.~we write $g_j(x,u)$ and $l(x,u)$, and define the state space as $\mathbf{X}:=\{x\in\mathbb{R}^{n_x}: g_{j}(x)\geq 0, j=1,\ldots,N_{g_x};x_c\in[0,T]\}$, the action space as $\mathbf{U}:=\{u\in\mathbb{R}^{n_u}: g_{j}(u)\geq 0, j=N_{g_x}+1,\ldots,N_{g_u}\}$, the feasible set as $\mathbf{C}:=\{(x,u)\in\mathbb{R}^{n_x}\times\mathbb{R}^{n_u}: g_{j}(x,u)\geq 0, j=1,\ldots,N_g;x_c\in[0,T-1]\}$ and the admissible controls as $\mathbf{U}(x):=\{u: (x,u)\in\mathbf{C}\}$. However, this restriction can be dropped in practice.

Consider the (nonstationary) Markov control model formed by the tuple $(\mathbf{X}_t,$$\allowbreak\mathbf{U}_t,$$\allowbreak\{\mathbf{U}(x)_t|x\in\mathbf{X}_t\},\allowbreak f_t(x,u),\allowbreak l_t(x,u),H(x))$ for which we wish to find an optimal control policy $\varrho^*$. Note that for the purposes of the derivations which follow, nonstationary Markov control models can be represented using an equivalent stationary model using state augmentation \citep[Section 1.3]{Hernandez1989}. Under Assumption \ref{as:invariance}, from \citep[Theorem 3.2.1]{HernandezLerma1970} there exists an optimal policy $\varrho$ that is deterministic and can therefore be expressed in the form $u=\varrho^*(x)$. The state-action occupation measure at time step $t$ for a given policy $\varrho$ and initial state measure $\nu_0$ is a Borel measure $\mu_t \in \mathcal{M}(\mathbf{C}_t)_+$ on the feasible set $\mathbf{C}_t$, defined by
\begin{equation}
\mu_t(B):= \mathbf{E}^{\varrho}_{\nu_0}(1_B(x,u))
\end{equation}
for all Borel sets $B$ of $\mathbf{C}_t$. $\mathbf{E}^{\varrho}_{\nu_0}$ is the expected value under policy $\varrho$ given some initial distribution $\nu_0$ of the state. Measure $\mu_t$ contains all information about the relationship between the state $x$ and control input $u$ (which depends on $x$) at time step $t$.

Let $\pi: \mathcal{M}(\mathbf{C}_t)_+ \rightarrow \mathcal{M}(\mathbf{X}_{t})_+$ be the projection from state-action space onto the state alone.\footnote{For any Borel measure $\mu_t \in \mathcal{M}(\mathbf{C}_t)_+$ this is formally defined by ($\pi\mu_t$)($\mathit{B}$) = $\mu_t$(($\mathbb{R}^{n_u}\times \mathit{B})~\cap\enskip \mathbf{C}_t)$ for all Borel subsets $B$ of $\mathbf{X}_t$.} Then the linear operator $\mathcal{L}_t:\mathcal{M}(\mathbf{C}_t)_+ \rightarrow \mathcal{M}(\mathbf{X}_{t+1})_+$ maps the state-action occupation measure at time step $t$ to the occupation measure projected onto the state space $\mathbf{X}_{t+1}$ at time step $t+1$ under the dynamics $f_t(x,u)$:
\begin{equation}\label{eq:foias}
\pi\mu_{t+1}(A)=\mathcal{L}_t\mu_{t}(A)=\int_{\mathbf{C}_t} 1_A(f_t(x,u))d\mu_t
\end{equation}
for all Borel sets $A$ of $\mathbf{X}_{t+1}$.\footnote{This operator was first defined in \cite{lasota1994}, and used for the infinite-horizon control application in \cite{Savorgnan2009}.} In words, the probability mass of the state distribution in set $A$ at time $t+1$ is equal to the total contributions of mass brought into $A$ by the dynamics, across all infinitesimal elements of the state-action distribution $\mu_t$. This operator therefore encodes consistency with the dynamics of successive state-action distributions ($\mu_t,\mu_{t+1}$).

%The state-action probability measures for all time steps can be summed to yield a single finite-horizon occupation measure $\mu:=\sum_{t=0}^{T-1}\mu_t$, which has mass $T$. All summands are supported on the same set $\mathbf{C}$ thanks to our use of an auxiliary state for the time index. Therefore the operator $\mathcal{L}$ can be applied equally well to the $T$-step measure $\mu$ as to any of the single step measures $\mu_t$.  

Using these definitions, the following linear constraint describes all state-action probability measures $\mu_0,\mu_1,\ldots,\mu_{T-1}$ that are consistent with a control policy $\varrho$, the dynamics $f_t(x,u)$, and a free choice of terminal state measure $(\nu_T \otimes \delta_T) \in \mathcal{M}(\mathbf{X}_T)_+$:
\begin{equation} \label{eq:measure_consistency}
\nu_0\otimes\delta_0+\sum_{t=0}^{T-1}\mathcal{L}_t\mu_t=\sum_{t=0}^{T-1}\pi\mu_t+\nu_{T}\otimes\delta_T\, .
\end{equation}
We use $\nu_t$ to denote a probability measure over all elements of vector $x$ except the auxiliary time index state $x_c$, and $\delta_t$ to denote the Dirac measure supported on $t$ for $x_c$. Thus, measure $\nu_0\otimes\delta_0$ is an initial probability distribution on $\mathbf{X}_0$, where $\delta_0$ accounts for $x_c$ being supported on $t=0$. Similarly, $\nu_{T}\otimes\delta_T$ is the terminal Borel probability measure on $\mathbf{X}_T$. Note that the sum of measures on each side of \eqref{eq:measure_consistency} is supported on $x_c = 0, 1, \ldots, T$, thus the single constraint encodes all $T$-step trajectories of the system.      

We can now formulate the \ac{GMP} \eqref{eq:gmp}, which is a $T$-step decision problem related to \eqref{eq:energystorage}. Measures $\mu_t$ and the terminal state measure $\nu_T$ fully specify the solution of \eqref{eq:energystorage} for a given distribution $\nu_0$ of the initial state $x_0$.
\begin{subequations}\label{eq:gmp}
	\begin{align}
	\rho^* :=\min_{\{\mu_t\}_{t=0}^{T-1},\,\nu_{T}} \quad &\sum_{t=0}^{T-1}\int_{\mathbf{C}_t} l_t(x,u)d\mu_t+\int_{\mathbf{X}_T} H(x) d(\nu_{T}\otimes\delta_T) \label{eq:gmp_obj}\\
	\text{s.t.}\quad &\nu_0\otimes\delta_0+\sum_{t=0}^{T-1}\mathcal{L}_t\mu_t=\sum_{t=0}^{T-1}\pi\mu_t+\nu_{T}\otimes\delta_T, \label{eq:gmp_dyn}\\
	&\mu_t \in \mathcal{M}(\mathbf{C}_t)_+,\,\, \nu_{T}\otimes\delta_T \in \mathcal{M}(\mathbf{X}_T)_+. \label{eq:gmp_cone}
	\end{align}
\end{subequations}
\begin{theorem}
	The optimal value $\rho^*$ of \eqref{eq:gmp} is equal to the optimal cost $V_0^*(x_0)$ of (\ref{eq:energystorage}) when $\nu_0$ is a Dirac measure on $x_0$, and equal to the expected value $\mathbf{E}_{\nu_0}(V_0^*(x_0))$ when $\nu_0$ is a probability measure.
\end{theorem}
\begin{proof}
	The finite-horizon problem (\ref{eq:gmp}), expressed as an equivalent stationary model \citep[Section 1.3]{Hernandez1989}, is a special case of the infinite horizon \ac{GMP} from \cite{HernandezLerma1970} and \cite{Savorgnan2009}. Problem (\ref{eq:energystorage}) can be restated as an infinite-horizon problem by setting the cost functions for $t>T$ to zero. Since the support of any $\mu_t$ is limited to values of auxiliary state $x_c$ on the interval $[0,T-1]$, by definition of the measure $\mu_t$, we have $\sum_{t=T+1}^{\infty}\pi{\mu}_t=0$ and $\sum_{t=T}^{\infty}\mathcal{L}_t{\mu}_t=0$. Thus, the infinite-horizon \ac{GMP} presented in \cite{Savorgnan2009} reduces to (\ref{eq:gmp}).
	Due to Assumption \ref{as:polynomialcompact} (which implies continuity of $l_t(x,u)$ and $f_t(x,u)$, and compactness of $\mathbf{C}_t$ and $\mathbf{X}_t$), we have $\rho^*=\mathbf{E}_{\nu_0}(V_0^*(x_0))$ by \citep[Theorem 6.3.7]{HernandezLerma1970}.
\end{proof}  

\subsection{Value function approximation}\label{dual}
To facilitate the decomposition approach in Section \ref{MomentSOS}, we  rewrite \eqref{eq:gmp} by introducing state measures $\nu_{t}\otimes\delta_t \in \mathcal{M}(\mathbf{X}_t)_+$ for $t=1,\ldots,T-1$, and replacing the single dynamical constraint \eqref{eq:gmp_dyn} with $T$ separate one-step constraints,
\begin{equation}\label{eq:gmp_dyn_alternative}
	\nu_t\otimes\delta_t+\mathcal{L}_t\mu_t=\pi\mu_t+	\nu_{t+1}\otimes\delta_{t+1},\quad t=0,\ldots,T-1.
\end{equation}
The resulting \ac{GMP} is equivalent to \eqref{eq:gmp}, since eliminating the measures $\nu_{t}\otimes\delta_t \in \mathcal{M}(\mathbf{X}_t)_+$ using equalities \eqref{eq:gmp_dyn_alternative} recovers constraint \eqref{eq:gmp_dyn}. We now state the dual of this equivalent \ac{GMP}, and show that the component of its solution for $t=0$ approximates the value function $V_0^*(x)$ of (\ref{eq:bellman}) over the initial distribution $\nu_0$.
Following the dualization process of \cite{anderson1987} for infinite-dimensional linear programs, we obtain \eqref{eq:dual}. This is another infinite-dimensional linear program, in this case in the space of bounded continuous functions on $\mathbf{X}_t$ for each time step $t$, denoted $\mathcal{C}(\mathbf{X}_t)$.
\begin{subequations}\label{eq:dual}
	\begin{align}
	\theta^* :=\max_{\{V_t\in \mathcal{C}(\mathbf{X}_t)\}_{t=0}^{T-1}} &\int_{\mathbf{X}_0} V_0(x)d(\nu_0\otimes\delta_0)\\
	\text{s.t.}\quad &l_t(x,u)-V_t(x)+V_t(f_t(x,u)) \geq 0,\quad\forall(x,u)\in\mathbf{C}_t,\quad t=0,\ldots,T-1, \label{eq:positivity_S}\\
	&V_{t+1}(x)\geq V_t(x),\quad\forall x \in\mathbf{X}_{t+1},\quad t=0,\ldots,T-2\label{eq:positivity_XN}, \\
    &H(x)\geq V_{T-1}(x), \quad\forall x \in\mathbf{X}_{T}. \label{eq:positivity_XNT}
	\end{align}
\end{subequations}
The integral $d(\nu_0 \otimes \delta_0)$ reflects the initial state distribution $\nu_0$ and initial value of the auxiliary state $x_c$, which is always 0. Thus the objective integrates $V_0(x)$ over a ``slice'' of $x$-space at $x_c = 0$. 

Note that each function $V_t(x)$ in \eqref{eq:dual} is constrained at time steps $t$ \emph{and} $t+1$, and that $V_0(x), \ldots, V_{T-1}(x), H(x)$ form a chain of coupled functions. Constraint \eqref{eq:positivity_S} is a relaxation of the Bellman optimality condition for each pair of points $(x, f(x,u))$ generated by an $(x,u)$ pair in $\mathbf{C}_t$; since $x$ and $f(x,u)$ have time index states $x_c = t$ and $x_{c} = t+1$ respectively, $V_t(x)$ is constrained in how it changes between time steps $t$ and $t+1$. Constraint \eqref{eq:positivity_XN} upper-bounds $V_t(x)$ by the value of the ``next'' value function $V_{t+1}(x)$, on $x$ values with time index $x_c = t+1$.

Since we have shown that the finite-horizon case is just a special case of the infinite-horizon formulation and Assumption \ref{as:polynomialcompact} holds, Problem \eqref{eq:dual} is in fact the \ac{LP} formulation of the dynamic programming problem for \eqref{eq:energystorage} and there is no duality gap between (\ref{eq:gmp}) and (\ref{eq:dual}) \citep[Theorem 6.3.8]{HernandezLerma1970}. It is straightforward to show\footnote{The optimal solutions $\hat{V}_t(x)$ of (\ref{eq:dual}) are \emph{subsolutions} of the Bellman equation (\ref{eq:bellman}), i.e.~$\hat{V}_t(x) \leq l_t(x,u)+\hat{V}_t(f_t(x,u))$ on $\mathbf{C}_t$ and $\hat{V}_{t}(x)\leq \hat{V}_{t+1}(x)$ on $\mathbf{X}_{t+1}$, with $\hat{V}_{T-1}(x)\leq H(x)$ on $\mathbf{X}_{T}$. As pointed out in \cite{Savorgnan2009}, this leads to the fact that $\hat{V}_0(x)$, a maximizer, minimizes the quantity $\int_{\mathbf{X}_0} |V_0^*(x)-\hat{V}_0(x)|d(\nu_0\otimes\delta_0)=\int_{\mathbf{X}_0}V_0^*(x)-\hat{V}_0(x)d(\nu_0\otimes\delta_0)$.} that for all feasible solutions of \eqref{eq:dual}, $V_t(x) \leq V_t^*(x)$ on $\mathbf{X}_{t}$ for $t=0,\ldots,T-1$.

\section{Moment DDP}\label{MomentSOS}
We now present an algorithm, termed Moment \ac{DDP}, to find approximate solutions to (\ref{eq:energystorage}) that are fitted to a probability distribution $\nu_0$ of values of $x_0$. This is achieved by decomposing the multi-stage problems \eqref{eq:gmp} and \eqref{eq:dual} into single stages and solving finite approximations of these problems. We first describe the backward recursion (Section \ref{backward}) and forward simulation (Section \ref{forward}), which are familiar concepts from existing \ac{DDP} approaches, and then state the Moment DDP algorithm as a whole in Section \ref{algorithm}.

Moment \ac{DDP} uses the same stage-wise decomposition principle as conventional \ac{DDP}, in that it simulates state trajectories in the forward simulation and then solves dual problems to generate lower-bounding functions in the backward recursion. However it is different in two important respects. First, the forward simulation consists of a sequence of single-stage problems over \emph{moments} of the occupation measure instead of the point values or sampled uncertainty realizations used in conventional (S)DDP. These moments are a finite approximation of the original problem \eqref{eq:gmp} over occupation measures. Second, the backward recursion, comprising dual \ac{SOS} problems, generates \emph{polynomial} rather than linear cuts, and under-approximates the value function most closely around the state distribution computed by the forward simulation. Analogously to conventional \ac{DDP}, the cuts are used in the forward simulation as approximate cost-to-go functions to improve the candidate state trajectory. The sum of costs in the forward simulation (as estimated from the truncated moment series) represents an upper bound on the optimal cost attainable under the moment/SOS approximation, while the expected value (with respect to the given initial state distribution $\nu_0$) of the value function obtained for $t=0$ represents a lower bound. The difference between the upper and lower bounds is used as a convergence criterion for terminating the algorithm.

Alongside our general description of Moment \ac{DDP}, we will use problem \eqref{eq:gmp} with horizon $T=2$ to illustrate the decomposition into single-stage problems. The proof of convergence in Section \ref{convergence} will also apply to the two-stage problem, with an induction argument used to extend this to arbitrary $T$.
% \begin{subequations}\label{eq:gmp_2t}
% 	\begin{align}
% 	\rho_{12}=\min_{\mu_0,\mu_1,\nu_{2}} &\int_{\mathbf{C}_0} l_0(x,u)d\mu_0+\int_{\mathbf{C}_1} l_1(x,u)d\mu_1+\int_{\mathbf{X}_2} H(x) d(\nu_{2}\otimes \delta_{2}), \label{eq:gmp2t_obj}\\
% 	\text{s.t.}\quad &\nu_0\otimes\delta_0+\mathcal{L}\mu_0+\mathcal{L}\mu_1=\pi{\mu_0}+\pi{\mu_1}+\nu_{2}\otimes \delta_{2}, \label{eq:gmp2t_dyn1}\\
% 	&\mu_0 \in \mathcal{M}(\mathbf{C}_0)_+, \mu_1 \in \mathcal{M}(\mathbf{C}_1)_+, \nu_{2}\otimes \delta_{2}\in \mathcal{M}(\mathbf{X}_2)_+. \label{eq:gmp2t_cones}
% 	\end{align}
% \end{subequations}
\subsection{The backward recursion}\label{backward}
The backward recursion creates a new polynomial lower bounding function $V_{t,z}(x)$ for the value function for $t = T-1,\ldots,0$, analogous to the Benders cuts in conventional \ac{DDP}. For each time step $t$ and iteration $z$, the single-stage subproblem uses the following data:
\begin{itemize}
\item The lower-bounding functions already generated from earlier  backward recursions (including the current one), $V_{t+1,i}(x)$, $i=0,\ldots,z$, satisfying $V_{t+1,i}(x) \leq V_{t+1}^*(x)$ for all $x \in \mathbf{X}_{t+1}$.
\item The state measure $\nu_t\otimes\delta_t\in\mathcal{M}(\mathbf{X}_t)$ from the last forward pass completed.
\end{itemize}
By the standard dynamic programming argument used in conventional \ac{DDP}, the subproblem corresponds to the first stage of a version of problem \eqref{dual} starting at step $t$:
 \begin{subequations}\label{eq:dual_1t}
	\begin{align}
	\theta_t:=\max_{V_{t,z}\in \mathcal{C}(\mathbf{X}_t)} &\int_{\mathbf{X}_t} V_{t,z}(x)d(\nu_t\otimes\delta_t) \label{eq:dual_1t_obj}\\
	\text{s.t.}\quad &l_t(x,u)-V_{t,z}(x)+V_{t,z}(f_t(x,u)) \geq 0,\quad\forall(x,u)\in\mathbf{C}_t,\label{eq:bwd_bellman}\\
	&V_{t,z}(x) \leq \left\{ \begin{array}{ll}\max\big\{V_{t+1,0}(x),\ldots, V_{t+1,z}(x)\}, \,\,  \forall x\in\mathbf{X}_{t+1}, & \text{if $t\in\{0,\ldots,T-2\}$} , \\
    H(x), \,\,  \forall x \in \mathbf{X}_{t+1}, & \text{if $t=T-1$.} \end{array}\right.\label{eq:bwd_bc}
	\end{align}
\end{subequations}
This problem is illustrated in Fig.~\ref{fig:backward_rec}. Constraint \eqref{eq:bwd_bellman} restricts the change in the value function from time step $t$ to time step $t+1$ according to the Bellman principle, and \eqref{eq:bwd_bc} upper-bounds the value function at time step $t+1$ by the lower bounds already derived for stage $t+1$ of the problem.
\begin{figure}
	\centering
	\includegraphics[trim={0 0cm 0 0},clip,scale=0.7]{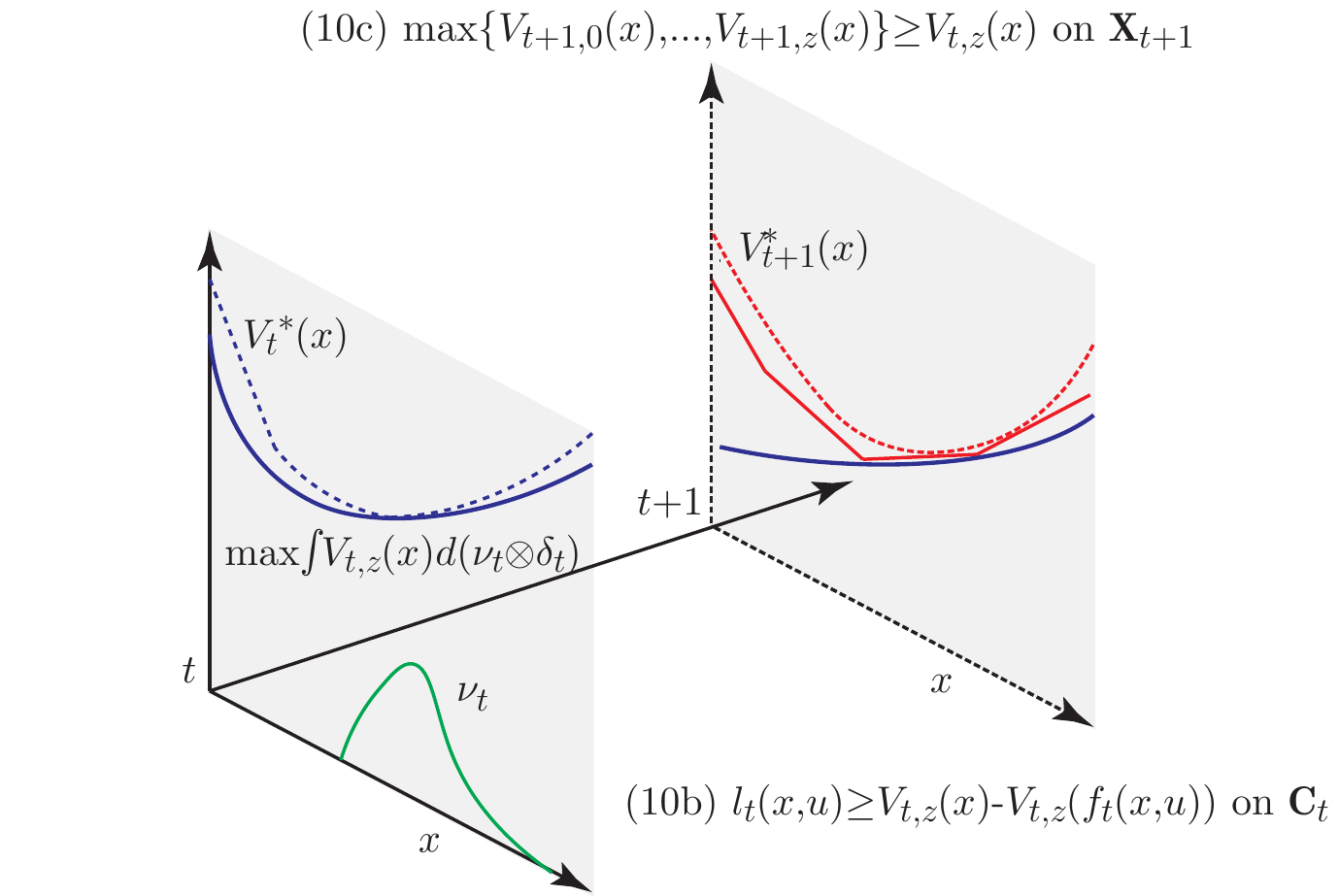}
	\caption[Illustration of the backward recursion]{Illustration of the infinite-dimensional \ac{LP} \eqref{eq:dual_1t}. The function $V_{t,z}(x)$ (blue) is maximized over the state distribution $\nu_t$ (green) at time step $t$ subject to constraints \eqref{eq:bwd_bellman} and \eqref{eq:bwd_bc}, in order to approximate the value function $V^*_t(x)$ (dashed blue). Constraint \eqref{eq:bwd_bellman} ensures $V_{t,z}(x) \leq V_t^*(x)$ by limiting the values of $V_{t,z}(x)$ at time step $t$ such that transitions to step $t+1$ incur costs that respect the Bellman inequality condition. Constraint \eqref{eq:bwd_bc} bounds $V_{t,z}(x)$ from above at $t+1$ by the point-wise maximum (red) of lower-bounding functions computed in previous iterations for time step $t+1$. These are in turn under-approximations of the optimal value function (dashed red) at $t+1$.}
\label{fig:backward_rec}
\end{figure}

Problem \eqref{eq:dual_1t} is intractable owing to its infinite-dimensional decision space, but can be approximated using a polynomial parameterization of $V_{t,z}(x)$. We note that, except for the case $t=T-1$, constraint \eqref{eq:bwd_bc} is equivalent to
\begin{equation*}
V_{t,z}(x) \leq y, \quad \forall(x,y) \in (\mathbf{X}_{t+1}\times \mathbb{R}) \cap \{(x,y) \,:\, y \geq V_{t+1,0}(x),\ldots, y \geq V_{t+1,z}(x)\} \, ;
\end{equation*}
this leads to the following \ac{SOS} program for each time step $t=T-1,\ldots,0$:
%  At $x_c = t+1$ it must lie below $V_{t+1}(x)$, i.e.,
% \begin{align}
% V_{t,z}(x) & \leq \max\big\{V_{t+1,1}(x),\ldots, V_{t+1,z}(x)\}, \quad \forall x\in\mathbf{X}_{t+1}, \nonumber \\
% \Leftrightarrow V_{t,z}(x) & \leq y, \quad \forall(x,y)\in\mathbf{X}_{t+1}\cap\{y \geq V_{t+1,1}(x),\ldots, y \geq V_{t+1,z}(x)\}.\label{eq:bwd_bc}
% \end{align}
% The new $V_{t,z}(x)$ must also be a subsolution of the Bellman equality. This is a property that couples $V_{t,z}(x)$ between $x_c = t$ and $x_c = t+1$, noting that $f(x,u)$ has time index $t+1$:
% \begin{equation}\label{eq:bwd_bellman}
% V_{t,z}(x) \leq l_t(x,u) + V_{t,z}(f(x,u)),\quad \forall(x,u) \in \mathbf{C}_t.
% \end{equation}
% If conditions \eqref{eq:bwd_bc} and \eqref{eq:bwd_bellman} hold, it follows directly that $V_{t,z}(x) \leq V_t^*(x)$ on $\mathbf{X}_t$ at $x_c=t$. 
\begin{subequations}\label{eq:sos_backward}
	\begin{align}
\theta_{t,z}:=\max_{\mathbf{V}_{t,z},\boldsymbol\sigma_{t,z}}\enskip&\langle\mathbf{V}_{t,z}, \mathbf{q}_{t,z}\rangle\\
	\text{s.t.}\quad &l_t(x,u)-V_{t,z}(x)+V_{t,z}(f_t(x,u)) = \textrm{Q}_k(\mathbf{C}_t)\label{eq:putinar_S_backward},\\
	&y-V_{t,z}(x)=\textrm{Q}_k(\mathbf{Y}_{t+1,z}),\label{eq:putinar_XN_backward}\\
	& \textrm{deg}(V_{t,z}) \kappa_t \leq 2k\label{eq:sos_deg_limit}.
	\end{align}
\end{subequations}
The polynomial $V_{t,z}(x)$ is represented by its vector of monomial coefficients $\mathbf{V}_{t,z}$, and the objective \eqref{eq:dual_1t_obj} can thus be expressed as $\langle \mathbf{V}_{t,z}, \mathbf{q}_{t,z}\rangle$, where $\mathbf{q}_{t,z}$ is a vector of moments of the state distribution $\nu_t \otimes \delta_t$ returned at step $t-1$ of the last forward pass completed.\footnote{In our proposed implementation, the first backward pass takes place before the first forward pass, hence the moments $\mathbf{q}_{t,0}$ of the state trajectory must be initialized. The uniform distribution may be an appropriate choice when no information about the optimal state trajectory is available \textit{a priori}.} The constraints \eqref{eq:putinar_S_backward}-\eqref{eq:putinar_XN_backward} convert \eqref{eq:bwd_bc}-\eqref{eq:bwd_bellman} into equality constraints using Putinar's Positivstellensatz \citep{Putinar} for compact semi-algebraic sets, in which the slacks are written as quadratic modules $\textrm{Q}_k(\mathbf{C}_t)$ and $\textrm{Q}_k(\mathbf{Y}_{t+1,z})$ that are non-negative by construction; see definition \eqref{eq:sosdef}.
The vector $\boldsymbol\sigma_{t,z}$ contains all coefficients of the \ac{SOS} polynomials introduced by the quadratic modules and is subject to additional \ac{LMI} constraints not shown explicitly here, ensuring that the coefficients form valid \ac{SOS} polynomials.\footnote{More precisely, $\boldsymbol\sigma_{t,z}$ is a concatenation of the vectorizations of the matrix of coefficients $\mathbf{P}$, as described in Section \ref{sec:prelim}, for all of the \ac{SOS} polynomials $\sigma_i$ within the quadratic modules $\textrm{Q}_k(\mathbf{C}_t)$ and $\textrm{Q}_k(\mathbf{Y}_{t+1,z})$.} Constraints \eqref{eq:putinar_S_backward} and \eqref{eq:putinar_XN_backward} are implemented by matching the coefficients of each monomial on either side, i.e., using linear equality constraints linking the elements of $\boldsymbol V_{t,z}$  and $\boldsymbol\sigma_{t,z}$. Since the definition of the quadratic module limits the degree of polynomial used to $2k$, and polynomials $V_t$ are composed with polynomials $f_t(x,u)$ in \eqref{eq:putinar_S_backward}, the degree of $V_t$ must be restricted by \eqref{eq:sos_deg_limit}, where $\kappa_t:=\max_{i=1,\ldots,n_x}(\textrm{\textrm{deg}}\, (f_{t,i}(x,u)))$ is the highest-order polynomial found in the dynamics. 

In constraint \eqref{eq:putinar_XN_backward} we introduced a new epigraph set $\mathbf{Y}_{t+1,z}$. For each time step $t=T,\ldots,1$, $\mathbf{Y}_{t,z}$ is defined by the $z$ lower-bounding functions generated so far for that time step, and an upper bound $\overline{y}$ on the epigraph variable $y$:
\begin{equation*}
\mathbf{Y}_{t,z}:=\left\{ \begin{array}{ll} \{(x,y):x \in \mathbf{X}_t; y\in\mathbb{R}; y \leq \overline{y};y\geq  V_{t,i}(x),\,i=0,\ldots,z\}, \, &t=1,\ldots,T-1, \\ \{(x,y):x \in \mathbf{X}_t; y\in\mathbb{R}; y \leq \overline{y}; y\geq  H(x)\}, \, & t=T. \end{array} \right.
\end{equation*}
The parameter $\overline{y}\in\mathbb{R}$ must be chosen in advance and ensures that, in combination with at least one lower-bounding value function, the epigraph set is compact.\footnote{We acknowledge that this is not an epigraph in the strict sense of the word, since it includes an upper bound on $y$. The value of $\overline{y}$ used to define $\mathbf{Y}_t$ must be larger than the greatest sum of costs from time steps $t$ to $T$ that can occur in any state trajectory. Since the state-input set is compact, the stage cost is bounded, and the number of stages is finite, it is generally straightforward to obtain such a bound.} %Lastly, we recall that $\kappa_t=\max_i(\textrm{\textrm{deg}}\, (f_{t,i}(x,u)))$, i.e., the highest polynomial degree found in the dynamics.

Since the function parameterization in \eqref{eq:sos_backward} is contained in the feasible set of \eqref{eq:dual_1t}, it follows that $\theta_{t,z}$ is upper bounded by the optimal value of \eqref{eq:dual_1t}. The approximation accuracy is known to improve as $k$ increases \citep{korda2016}.

Returning to the two-stage example, the backward recursion at iteration $z$ for $t=1$ is a \ac{SOS} problem of type (\ref{eq:sos_backward}):
\begin{subequations}\label{eq:sos_backward_2t}
	\begin{align}
	\theta_{1,z}= \max_{\mathbf{V}_{1,z},\boldsymbol\sigma_{1,z}}\enskip& \langle\mathbf{V}_{1,z}, \mathbf{q}_{1,z}\rangle\\
	\text{s.t.}\quad &l_1(x,u)-V_{1,z}(x)+V_{1,z}(f_1(x,u)) =\textrm{Q}_k(\mathbf{C}_1)\label{eq:putinar_S_backward_2t},\\
	&H(x)-V_{1,z}(x) =\textrm{Q}_k(\mathbf{X}_2), \label{eq:putinar_XN_backward_2t}\\
	&\textrm{deg}(V_{1,z}) \kappa_1\leq 2k,\label{eq:deglim_t2}
	\end{align}
\end{subequations}
We add the optimal solution $\hat{V}_{1,z}$ of (\ref{eq:sos_backward_2t}) to the epigraph set $\mathbf{Y}_{1,z}$ and solve a \ac{SOS} problem for $t=0$:
\begin{subequations}\label{eq:sos_backward_1t}
	\begin{align}
	\theta_{0,z} = \max_{\mathbf{V}_{0,z},\boldsymbol\sigma_{0,z}}\enskip & \langle\mathbf{V}_{0,z}, \mathbf{q}_{0,z}\rangle\\
	\text{s.t.}\quad &l_0(x,u)-V_{0,z}(x)+V_{0,z}(f_0(x,u)) =\textrm{Q}_k(\mathbf{C}_0)\label{eq:putinar_S_backward_1t},\\
	&y-V_{0,z}(x)=\textrm{Q}_k(\mathbf{Y}_{1,z}),\\
	&\textrm{deg}(V_{0,z}) \kappa_0 \leq 2k\label{eq:deglim_t1}.
	\end{align}
\end{subequations}
In the Moment DDP algorithm described in Section \ref{algorithm}, the lower bound value $\theta_{LB,z}=\theta_{0,z}$ is used in the termination criterion.

%%%%%%%%%%%%%%%%%%%%%%%%%%%%%%%%%%%%%%%%%%%%%%%%%%%%%

\subsection{The forward simulation}\label{forward}
The forward simulation finds, for each $t=1,\ldots,T$, an approximate solution to a single stage of the \ac{GMP} \eqref{eq:gmp}, in which the state occupation measure $\nu_t$ is inherited from the previous step's solution, and the cost-to-go is under-approximated by the lower-bounding functions $V_{t+1,i}(x)$ generated in the backward recursions completed so far:
\begin{subequations}\label{eq:gmp_1stage}
	\begin{align}
	\rho_t:=\min_{\mu_{t},\nu_{t+1}} \,\, &\int_{\mathbf{C}_t} l_t(x,u)d\mu_{t}+\int_{\mathbf{X}_{t+1}} \max_{i=0,\ldots,z-1} V_{t+1,i}(x) d(\nu_{t+1}\otimes \delta_{t+1}), \label{eq:gmp1s_obj}\\
	\text{s.t.}\quad &\nu_t\otimes\delta_t+\mathcal{L}\mu_t=\pi{\mu_t}+\nu_{t+1}\otimes \delta_{t+1}, \label{eq:gmp1s_dyn1}\\
	&\mu_t \in \mathcal{M}(\mathbf{C}_t)_+, \nu_{t+1}\otimes \delta_{t+1}\in \mathcal{M}(\mathbf{X}_{t+1})_+. \label{eq:gmp1s_cones}
	\end{align}
\end{subequations}
As this problem is infinite-dimensional and therefore intractable, the approximation used is an optimization over a finite vector of moments of the state-action occupation measure $\mu_t$ at time step $t$, and the state occupation measure $\nu_{t+1}$ at time step $t+1$.

We now explain how this finite-moment approximation of \eqref{eq:gmp_1stage} is represented. Let $\mu_{t,z}$ be the state-action occupation measure on $\mathbf{C}_t$ for a single time step $t$ at iteration $z$, and let $m_{t,z}^{\alpha\gamma}$ be its $(\alpha, \gamma)$ moment for non-negative integer vectors $\alpha \in \mathbb{N}^{n_x}$ and $\gamma \in \mathbb{N}^{n_u}$, defined by
\begin{equation}\label{eq:measure_moments}
m_{t,z}^{\alpha\gamma}:=\int_{\mathbf{C}_t} x^\alpha u^\gamma  d\mu_{t,z} \,.
\end{equation}
Following convention from related literature, the vector-valued exponents are interpreted as $x^\alpha=x_1^{\alpha_1} x_2^{\alpha_2}\ldots x_{n_x}^{\alpha_{n_x}}$ and $u^\gamma=u_1^{\gamma_1} u_2^{\gamma_2}\ldots u_{n_u}^{\gamma_{n_u}}$, with $\sum_{i=1}^{n_x}\alpha_{i}+\sum_{i=1}^{n_u} \gamma_{i}\leq 2k$. 

We use the epigraph set $\mathbf{Y}_{t+1,z-1}$ created in the previous backward recursion to accommodate the maximum in the second term of \eqref{eq:gmp1s_obj}. For each time step $t=1,\ldots,T$ and iteration $z$, we define the moments of the augmented state measure $\nu_{t,z}\otimes\delta_{t}$ supported on the epigraph set $\mathbf{Y}_{t,z-1}$:
\begin{equation}\label{eq:terminal_moments}
q^{\alpha\eta}_{t,z}:=\int_{\mathbf{Y}_{t,z-1}} x^\alpha y^\eta  d(\nu_{t,z}\otimes\delta_{t}),
\end{equation}
where $x^\alpha=x_1^{\alpha_1} x_2^{\alpha_2}\ldots x_{n_x}^{\alpha_{n_x}}$ and $y$ is the scalar epigraph variable used in the definition of $\mathbf{Y}_{t,z-1}$, with $\sum_{i=1}^{n_x}\alpha_{i}+\eta\leq 2k$. We collect these moments into vectors $\mathbf{m}_{t,z}$ and $\mathbf{q}_{t,z}$ respectively for each iteration $z$ of the \ac{DDP} algorithm. The number of elements in $\mathbf{m}_{t,z}$ is combinatorial, given by $n_\mathbf{m}=\binom{n_x+n_u+k}{k}$. Similarly, the vector $\mathbf{q}_{t,z}$ has size $n_\mathbf{q}=\binom{n_x+1+k}{k}$. The moments (${q}_{t+1,z}^{\alpha 0}$) of the state distribution at time step $t$, recalling that the superscript $0$ signifies that $y$ is excluded, are used as initial conditions in time step $t+1$.

As with conventional \ac{DDP}, the forward problem in Moment \ac{DDP} for each stage $t=0,\ldots,T-1$ is dual to the backward problem \eqref{eq:sos_backward}. It takes the form of a \ac{SDP} in terms of the moments (up to degree $2k$) of $\mu_{t,z}$ and $\nu_{t+1,z}\otimes\delta_{t+1}$:
\begin{subequations}\label{eq:sdp_forward}
	\begin{align}
	\rho_{t,z}:=\min_{\mathbf{m}_{t,z},\mathbf{q}_{t+1,z}}\enskip &L_{\mathbf{m}_{t,z}}(l_t)+L_{\mathbf{q}_{t+1,z}}(y)\label{eq:primalsdpobj_f}\\
	\text{s.t.}\quad &L_{\mathbf{m}_{t,z}}\Big(x^\alpha-f_t(x,u)^{\alpha}\Big)+q_{{t+1},z}^{\alpha 0}=
	q_{{t,z}}^{\alpha 0}, \enskip \alpha \in \mathbb{N}^{n_x}, \sum_{i=1}^{n_x} \alpha_{i} \leq \lfloor 2k/ \kappa_t\rfloor,\label{eq:primalsdpdyn_f}\\
	&M_{k-d_{g_{t,j}}}(g_{t,j}\mathbf{m}_{t,z})\succeq0, \quad j={1,\ldots,N_{g,t}},\label{eq:primalsdplocal2_f}\\
	&M_{k-d_{v_{t+1,s}}}(v_{t+1,s}\mathbf{q}_{t+1,z})\succeq0, \quad s={1,\ldots,N_{g_{x},t}+z+1},\label{eq:primalsdpmax_f}\\
	&M_k(\mathbf{m}_{t,z})\succeq 0, M_k(\mathbf{q}_{t+1,z})\succeq 0,\label{eq:primalsdppsd_f}
	\end{align}
\end{subequations}
where $f(x,u)^{\alpha}$ is shorthand for $f_1(x,u)^{\alpha_1}f_2(x,u)^{\alpha_2}\allowbreak\ldots f_{n_x}(x,u)^{\alpha_{n_x}}$. 

In brief, the objective (\ref{eq:primalsdpobj_f}) approximates the expected cost $\mathbf{E}_{\mu_{t,z}}(l_t)+\mathbf{E}_{\nu_{t+1,z}}(y)$ as a linear combination of moments of $\mu_{t,z}$ and $\nu_{t+1,z}$. The constraint (\ref{eq:primalsdpdyn_f}) represents a truncated form of the infinite-dimensional constraint (\ref{eq:gmp_dyn}), which means that the state update equation is transformed into a set of linear equalities on the moments of the state-action measure $\mu_{t,z}$ and state measure $\nu_{t+1,z}\otimes\delta_{t+1}$.
Constraints \eqref{eq:primalsdplocal2_f} and \eqref{eq:primalsdpmax_f} jointly represent ``moment relaxations'' of the support constraints \eqref{eq:gmp1s_cones} on $\mu_{t,z}$ and $\nu_{t+1,z}$, and constraints \eqref{eq:primalsdppsd_f} are used to ensure that the moment vectors are compatible with valid measures. We now explain the elements of (\ref{eq:sdp_forward}) in detail. 

The operator $L_{\mathbf{m}_{t,z}}:\mathbb{R}[x,u]\rightarrow\mathbb{R}$ is a linear mapping associated with a measure $\mu_{t,z}$ acting on a polynomial $h\in\mathbb{R}[x,u]$: 
\begin{equation}
L_{\mathbf{m}_{t,z}}(h):=\sum_{\alpha\gamma}h^{\alpha\gamma}m_{t,z}^{\alpha\gamma},
\end{equation}
where $m_{t,z}^{\alpha\gamma}$ are the moments of $\mu_{t,z}$ as defined in (\ref{eq:measure_moments}) and $h^{\alpha\gamma}$ represents the polynomial coefficient of $x^\alpha u^\gamma$, with vectors $\alpha$ and $\gamma$ interpreted in the same manner as for \eqref{eq:measure_moments}. Analogously, $L_{\mathbf{q}_{t,z}}:\mathbb{R}[x,y]\rightarrow\mathbb{R}$ is a linear mapping associated with the moments defined in \eqref{eq:terminal_moments}:
\begin{equation}
L_{\mathbf{q}_{t,z}}(h):=\sum_{\alpha\eta}h^{\alpha\eta}q_{t,z}^{\alpha\eta}.
\end{equation}
These operators are used to approximate the expected cost \eqref{eq:gmp1s_obj} in terms of moments, so that $\mathbf{E}_{\mu_{t,z}}(l_t)+\mathbf{E}_{\nu_{t+1,z}}(y) =\int_{\mathbf{C}_t} l_t d\mu_{t,z}+\int_{\mathbf{Y}_{t+1,z-1}} y d\nu_{t+1,z}$ becomes $L_{\mathbf{m}_{t,z}}(l_t)+L_{\mathbf{q}_{t+1,z}}(y)=\sum_{\alpha\gamma}l_t^{\alpha\gamma}{m}_{t,z}^{\alpha\gamma}+{q}_{t+1,z}^{01}$.

The same linear operator is used in constraint \eqref{eq:primalsdpdyn_f} to enforce consistency of the change in moments from $q_{t,z}^{\alpha 0}$, which are fixed data from the previous stage, and $q_{t+1,z}^{\alpha 0}$ under the dynamics.

The standard \emph{moment matrices} $M_k(\mathbf{m}_{t,z})$ and $M_k(\mathbf{q}_{t,z})$ of degree $k$ in (\ref{eq:primalsdppsd_f}); and the localizing matrices $M_{k-d_{g_{t,j}}}(g_{t,j}\mathbf{m}_{t,z})$ and $M_{k-d_{v_{t+1,s}}}(v_{t+1,s}\mathbf{q}_{t+1,z})$ in \eqref{eq:primalsdplocal2_f}-\eqref{eq:primalsdpmax_f} enforce a condition that ensures the generic vectors of moments are consistent with finite Borel measures on compact set.\footnote{In fact, this is a relaxation of the consistency condition, which is only guaranteed to hold for an infinite series of moments \citep[Theorem 3.8]{Lasserre2014}.} They are derived by applying the linear mappings $L_{\mathbf{m}_{t,z}}$ and $L_{\mathbf{q}_{t,z}}$ to the square of any polynomial $h$ of degree $k$:
\begin{equation}
L_{\mathbf{m}_{t,z}}(h^2)=\mathbf{h}^\top M_k(\mathbf{m}_{t,z}) \mathbf{h}\geq 0,\quad L_{\mathbf{q}_{t,z}}(h^2)=\mathbf{h}^\top M_k(\mathbf{q}_{t,z}) \mathbf{h}\geq 0,
\end{equation}
where $\mathbf{h}$ is the vector of coefficients of $h$. Thus, the moment matrix, which is linear in the elements of $\mathbf{m}_{t,z}$ or $\mathbf{q}_{t,z}$, is constrained to be a symmetric positive semi-definite matrix; the two constraints of \eqref{eq:primalsdppsd_f} are therefore standard \ac{LMI} constraints.

For notational convenience, we now write the constraints defining the epigraph set $\mathbf{Y}_{t,z-1}$ as $v_{t,s}(x,y)\geq 0,s=1,\ldots,N_{g_x}+z+1$. The \emph{localizing matrices} (\ref{eq:primalsdplocal2_f}) and (\ref{eq:primalsdpmax_f}), which are also standard in moment problems, enforce a moment relaxation of the support constraints $g_{t,j}(x,u) \geq 0$ (which define set $\mathbf{C}_t$) and $v_{t+1,s}(x,u) \geq 0$ (which define set $\mathbf{Y}_{t+1,z}$). These are positive semi-definite and of the form
\begin{equation}
\begin{aligned}
	&L_{\mathbf{m}}(g_{t,j} h^2)=\mathbf{h}^\top M_{k-d_{g_{t,j}}}(g_{t,j}\mathbf{m}_{t,z}) \mathbf{h}\geq 0,\\
	&L_{\mathbf{q}}(v_{t+1,s} h^2)=\mathbf{h}^\top M_{k-d_{v_{t+1,s}}}(v_{t+1,s}\mathbf{q}_{t+1,z}) \mathbf{h}\geq 0,
\end{aligned}
\end{equation}
where $d_{g_{t,j}}=\lceil \textrm{deg}(g_{t,j})/2 \rceil$ and $d_{v_{t+1,s}}=\lceil \textrm{deg}(v_{t+1,s})/2 \rceil$. 

%It is well known (\cite{Lasserre2014}) \joe{please provide section/theorem} that the truncated sequences of moments $\mathbf{m}_{t}$ and $\mathbf{q}_{t+1}$ have representing finite Borel measures $\mu_t$ on $\mathbf{C}_t$ and $\nu_{t+1}\otimes\delta_{t+1}$ on $\mathbf{X}_{t+1}$, if and only if the associated moment matrices and localizing matrices are positive semi-definite. 

%Based on this fundamental result \joe{Doesn't really seem to follow from what was just said above about moment and localizing matrices...}, 

Thus, \eqref{eq:sdp_forward} is a relaxation of \eqref{eq:gmp_1stage}, in which each of the constraints has been enforced on only a finite series of moments of $\mu_{t,z}$ and $\nu_{t+1,z}$. It therefore attains a lower optimal value than \eqref{eq:gmp_1stage}; recall that its dual, the \ac{SOS} program \eqref{eq:sos_backward}, is a restriction of the infinite-dimensional \ac{LP} shown in Fig.~\ref{fig:backward_rec} and has a corresponding lower optimal value.

We now state a known result concerning the value of relaxation (\ref{eq:sdp_forward}) as the order $k$ is increased:

\begin{lemma}\label{re:momentconv}
	Let Assumption \ref{as:polynomialcompact} hold, and let the feasible set $\mathbf{C}_t$ and epigraph set $\mathbf{Y}_{t+1,z}$ satisfy Putinar's condition \footnote{One can ensure that the sets $\mathbf{C}_t$ and $\mathbf{Y}_{t+1,z}$ satisfy Putinar's condition (see Definition 3.4 in \cite{Lasserre2007}) by including an additional ball constraint. For instance one can add $g_{N_g+1}(x,u)=R^2-\sum_i^{n_x}x_i^2-\sum_i^{n_u}u_i^2\geq 0$ with $R\in\mathbb{R}$ to the definition of $\mathbf{C}_t$. The assumption that $\mathbf{C}_t$ and $\mathbf{Y}_{t+1,z}$ are both compact makes it straightforward to determine such an $R$ in most cases.}. If $\rho_{t}$ is the optimal solution of the infinite-dimensional \ac{GMP} (\ref{eq:gmp_1stage}) at time step $t$, then as  $k\rightarrow\infty$ the optimal value of \eqref{eq:sdp_forward} approaches $\rho_t$ asymptotically from below.
\end{lemma} 
\begin{proof}
	Following Theorem 1 in \cite{Savorgnan2009}, one can show that $\rho_{t,z}$, when evaluated for increasing values of the relaxation degree $k$ used in constraint \eqref{eq:primalsdpdyn_f}, is a monotone non-decreasing sequence converging to $\rho_{t}$. This makes use of Putinar's Positivstellensatz, and the fact that measures on compact sets are uniquely determined by their infinite sequence of moments.
\end{proof}

In case of example \eqref{eq:gmp} with $T=2$, we start the forward simulation by solving a moment relaxation of degree $2k$ for $t=0$, a \ac{SDP} of type (\ref{eq:sdp_forward}) that includes the epigraph set $\mathbf{Y}_{1,z-1}$ built from all the value function under-approximators $\{V_{1,i}(x)\}_{i=0}^{z-1}$:
\begin{subequations}\label{eq:sdp_forward_first}
	\begin{align}
	\rho_{0,z}=\min_{\mathbf{m}_{0,z},\mathbf{q}_{1,z}}\enskip &L_{\mathbf{m}_{0,z}}(l_0)+L_{\mathbf{q}_{1,z}}(y)\label{eq:primalsdpobj_first}\\
	\text{s.t.}\quad & L_{\mathbf{m}_{0,z}}\Big(x^\alpha-f_0(x,u)^{\alpha} \Big)+{q}_{1,z}^{\alpha 0}=
	{q}_{0}^{\alpha 0},\enskip \alpha \in \mathbb{N}^{n_x}, \sum_{i=1}^{n_x} \alpha_{i} \leq \lfloor 2k/\kappa_0\rfloor,\label{eq:primalsdpdyn_first}\\
	&M_{k-d_{g_{0,j}}}(g_{0,j}\mathbf{m}_{0,z})\succeq0, \quad j={1,\ldots,N_{g,t}},\label{eq:primalsdplocal2_first}\\
	&M_{k-d_{v_{1,s}}}(v_{1,s}\mathbf{q}_{1,z})\succeq0, \quad s={1,\ldots,N_{g_{x},t}+z+1},\label{eq:primalsdpmax_first}\\
	& M_k(\mathbf{m}_{0,z})\succeq 0, M_k(\mathbf{q}_{1,z})\succeq 0,\label{eq:primalsdppsd_first}
	\end{align}
\end{subequations}
where we note that moments ${q}_{0}^{\alpha 0}$ (defined in the same way as \eqref{eq:terminal_moments}) are fixed data derived from the initial state distribution $\nu_0\otimes\delta_0$. If (\ref{eq:sdp_forward_first}) and (\ref{eq:sos_backward_1t}) are strictly feasible, there is no duality gap and $\rho_{0,z}=\theta_{LB,z}$.

The primal problem for $t=1$ is a moment relaxation with the optimal solution $\hat{\mathbf{q}}_{1,z}$ of (\ref{eq:sdp_forward_first}) as input data:
\begin{subequations}\label{eq:sdp_forward_2t}
	\begin{align}
	\rho_{1,z}=\min_{\mathbf{m}_{1,z},\mathbf{q}_{2,z}}\enskip &L_{\mathbf{m}_{1,z}}(l_1)+L_{\mathbf{q}_{2,z}}(H)\label{eq:primalsdpobj_2t}\\
	\text{s.t.}\quad & L_{\mathbf{m}_{1,z}}\Big(x^\alpha-f_1(x,u)^{\alpha}\Big)+{q}_{2,z}^{\alpha 0}= \hat{q}_{1,z}^{\alpha 0},\enskip \alpha \in \mathbb{N}^{n_x}, \sum_{i=1}^{n_x} \alpha_{i} \leq \lfloor 2k/ \kappa_1\rfloor,\label{eq:primalsdpdyn_2t}\\
	&M_{k-d_{g_{1,j}}}(g_{1,j}\mathbf{m}_{1,z})\succeq0, \quad j={1,\ldots,N_{g,t}},\label{eq:primalsdplocal2_2t}\\
	&M_{k-d_{v_{2,s}}}(v_{2,s}\mathbf{q}_{2,z})\succeq0, \quad s={1,\ldots,N_{g_{x},t}},\label{eq:primalsdpmax_2t}\\
	&M_k(\mathbf{m}_{1,z})\succeq 0, M_k(\mathbf{q}_{2,z})\succeq 0,\label{eq:primalsdppsd_2t}
	\end{align}
\end{subequations}
The updated moments $\hat{\mathbf{q}}_{1,z}$ computed by (\ref{eq:sdp_forward_first}) can then be used in a subsequent backward recursion to generate a new approximate value function in the backward recursion. If (\ref{eq:sos_backward_2t}) and (\ref{eq:sdp_forward_2t}) are strictly feasible, there is no duality gap and $\rho_{1,z}=\theta_{1,z}$. Using the optimal values of (\ref{eq:sdp_forward_first}) and (\ref{eq:sdp_forward_2t}), we define the upper bound as $\rho_{UB,z}=L_{\hat{\mathbf{m}}_{0,z}}(l_0)+L_{\hat{\mathbf{m}}_{1,z}}(l_1)+L_{\hat{\mathbf{q}}_{2,z}}(H)$ for use in the termination criterion of the algorithm described below. 

\subsection{Moment DDP algorithm} \label{algorithm}
Moment \ac{DDP} is stated formally in Algorithm \ref{alg:moment_sos}, and we now remark on some aspects of its implementation. 

Firstly, we note that the degree of the under-approximating value functions can in practice be chosen to be relatively low, since a single function need not be an active bound over the entire state space. This is illustrated in Fig.~\ref{fig:proofillustration}\if\thesismode0 in the Appendix\fi, which shows the lower-bounding functions generated by a sequence of six backward recursions for the single storage example of Section \ref{numerics}, alongside the approximation generated by discretized \ac{DP}. 

\begin{algorithm}[t]
	\setstretch{1}
	\caption{Moment \ac{DDP}}\label{alg:moment_sos}
	\textbf{Input:} Horizon $T$, functions $f_t(x,u)$, $l_t(x,u)$, $H(x)$, $g_{t,j}(x,u)$, tolerance $\epsilon$, initial moments $\mathbf{q}_{t,0}$\\
	\textbf{Output:} Upper bound $\rho_{UB,z}$, lower bound $\theta_{LB,z}$, epigraph sets $\mathbf{Y}_{t,z}$, trajectory moments $\mathbf{q}_{t,z}$\\
	\textbf{Indices:} Iteration $z$, time step $t$\\
	\begin{algorithmic}[1]
		\State $z\leftarrow 0$
        \State Create set $\mathbf{Y}_{T,0}$ parameterized by $H(x)$
        \For{$t=T-1,\cdots,1$} \Comment{Initial backward recursion: Section \ref{backward}}
		\State Solve (\ref{eq:sos_backward}) to obtain $V_{t,0}(x)$
		\State Create set $\mathbf{Y}_{t,0}$ parameterized by $V_{t,0}(x)$.
		\EndFor
		\Repeat \Comment{Repeat procedure until predefined tolerance $\epsilon$ is achieved}
		\State $z\leftarrow z+1$
		\For{$t=0,\cdots,T-1$} \Comment{Forward simulation: Section \ref{forward}}
		\State Solve (\ref{eq:sdp_forward}) to obtain state moments $\mathbf{q}_{t+1,z}$
		\EndFor
		\State Compute $\rho_{UB,z}=\sum_{t=0}^{T-1}{L}_{\hat{\mathbf{m}}_{t,z}}(l_t)+{L}_{\hat{\mathbf{q}}_{T,z}}(H)$ (optimal values of (\ref{eq:sdp_forward}))
        \For{$t=T-1,\cdots,0$} \Comment{Backward recursion: Section \ref{backward}}
		\State Solve (\ref{eq:sos_backward}) to obtain $V_{t,z}(x)$
        \State $\mathbf{Y}_{t,z} \leftarrow \mathbf{Y}_{t,z-1} \cap \{(x,y):y \geq V_{t,z}(x)\}$
		\EndFor
		\State Set $\theta_{LB,z}=\theta_{0,z}$ (optimal value of (\ref{eq:sos_backward}) for $t=0$)
		\Until{$\rho_{UB,z}-\theta_{LB,z}<\epsilon$}\label{basic_iter}
	\end{algorithmic}
\end{algorithm}

\if\thesismode1
\begin{figure}
	\centering
	\includegraphics[trim={0 1cm 0 0},clip,scale=1]{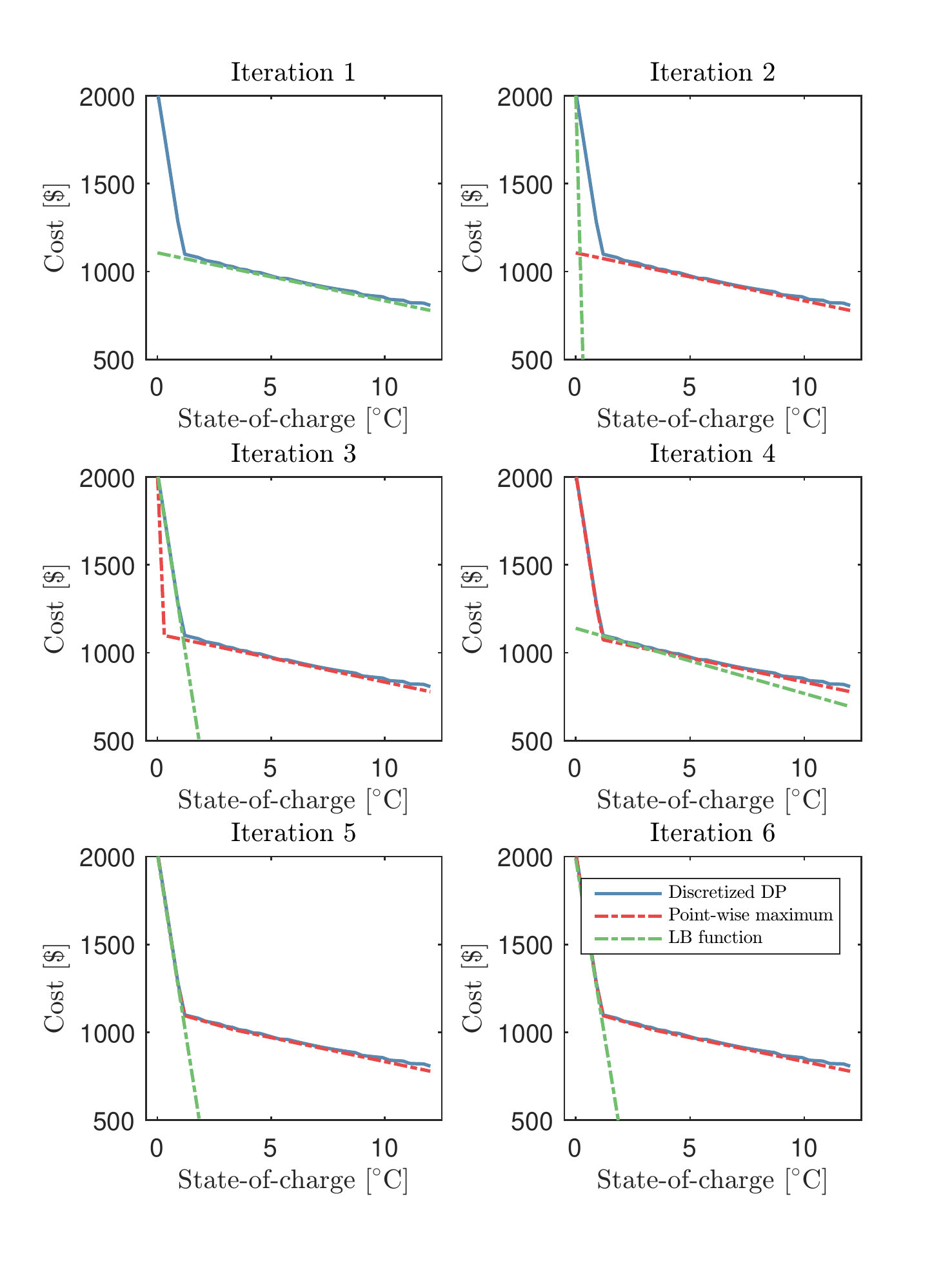}
	\caption[Sequence of six backward recursions for the single storage example]{Single storage example of Section \ref{numerics}: Cost of stored energy in the beginning of April using affine basis functions and $2k=4$, shown for six backward recursions.}\label{fig:proofillustration}
\end{figure}
\fi

Secondly, it can be attractive to preserve convexity of the lower-bounding functions added in the backward recursion, in order to reduce the cost of computing forward control actions. Following the approach of \cite{thanh2013}, convexity can be imposed on polynomials by constraining the Hessian of the value function in \eqref{eq:sos_backward} and adding additional variables to the primal \eqref{eq:sdp_forward}. This may of course cause an additional reduction in the tightness of the value function approximation. %In the case that the approximate value functions are \ac{SOS}-convex polynomials, a class that includes linear and convex quadratic functions, the projection of the relaxed epigraph set $\tilde{\mathbf{Y}}_{1,z}$ onto its first order moments describes the same set as the epigraph set $\mathbf{Y}_{1,z}$ \citep{Lasserre2015}.

Thirdly, if the problem input data remains constant over multiple time steps $t$, it becomes relatively straightforward to adapt a single stage of the forward and backward recursions in Algorithm \ref{alg:moment_sos} to span these steps. In this case, one can use a single polynomial to approximate a value function over the relevant interval on the time coordinate $x_c$. Value functions can then be extracted for a time step within a stage by setting $x_c$ to the relevant value. Throughout this \if\thesismode1 chapter\else paper\fi, however, we maintain equivalence between problem stages and time steps $t$ in \eqref{eq:energystorage} for clarity of notation.

\subsection{Extension to stochastic dynamics}

The Moment \ac{DDP} approach can be extended to stochastic polynomial dynamics, in which the state update is described by a function $f_t(x,u,w)$, without increasing the computational complexity significantly. Vector $w$ denotes an independent disturbance following the distribution $\omega_t$ supported on $\mathbf{W}_t$, of which the statistical moments can be computed; and entering polynomially into the state update.

If these conditions hold, moment and \ac{SOS} relaxations can be formulated using the same procedure described for generic optimal control problems in \cite{Savorgnan2009}. Specifically, the operator $\mathcal{L}_t$ is replaced by a new linear operator $\tilde{\mathcal{L}}_t:\mathcal{M}(\mathbf{C}_t)_+ \rightarrow \mathcal{M}(\mathbf{X}_{t+1})_+$ defined as
\begin{equation}\label{eq:stochfoias}
\pi\mu_{t+1}(A)=\tilde{\mathcal{L}}_t\mu_{t}(A):=\int_{\mathbf{C}_t}\int_{\mathbf{W}_t} 1_A(f_t(x,u,w))d\omega_t d\mu_t,
\end{equation}
for all Borel sets $A$ of $\mathbf{X}_{t+1}$. For simplicity of exposition, however, we have excluded stochastic dynamics from the derivations and numerical examples in the present \if\thesismode1 chapter\else paper\fi, and the only uncertainty we include arises from the initial state distribution.

\section{Convergence properties}\label{convergence}
In this section, we analyze the convergence of Algorithm \ref{alg:moment_sos} using an instance of the \ac{GMP} \eqref{eq:gmp} with $T=2$, and argue subsequently that the results extend to longer horizons. Lemma \ref{lemma:tightening} states that if the upper bound is strictly larger than the lower bound, (a relaxation of) the epigraph set strictly tightens from one iteration to the next. Lemmas \ref{lemma:upperboundbound} and \ref{lemma:lowerboundbound} bound the values of $\theta_{LB,z}$ and $\rho_{UB,z}$ used in the termination criterion. Finally, Theorem \ref{th:convergence} concludes that the Moment \ac{DDP} approach converges in finite iterations for any tolerance $\epsilon > 0$.
% \begin{figure}
% 	\centering
% 	\includegraphics[trim={0 0 0 0},clip,scale=0.9]{proofillustration2_mod.pdf}
% 	\caption[Illustration of the key convergence properties of the Moment DDP approach]{Illustration of the key convergence properties of the Moment \ac{DDP} approach. Lemma \ref{lemma:underapproximation} establishes that $\theta_{LB,z}$ and $\rho_{UB,z}$ are lower and upper bounds to the optimal value $\rho^*_k$. Lemma \ref{lemma:tightening} implies that the lower bound $\theta_{LB,z}$ is a monotonically increasing sequence. Theorem \ref{th:convergence} concludes that the Moment \ac{DDP} approach converges in a finite number of steps for a tolerance $\epsilon$.}\label{fig:proofillustration2}
% \end{figure}

% Let $\theta^*_k$ be the optimal value of the undecomposed \ac{SOS} approximation of \eqref{eq:dual} with $T=2$, dual to \eqref{eq:undecomposedsdp}:
% \begin{equation}\label{eq:undecomposedsos}
% \begin{aligned}
% \theta^*_k=&\max_{\mathbf{V}_{0},\mathbf{V}_{1},\boldsymbol\sigma_{0},\boldsymbol\sigma_{1}}\enskip  \langle\mathbf{V}_{0}, \mathbf{q}_{0}\rangle\\
% &\text{s.t.}\quad\textrm{\eqref{eq:putinar_S_backward_1t},\eqref{eq:deglim_t1}},\textrm{\eqref{eq:putinar_S_backward_2t}-\eqref{eq:deglim_t2}}, V_{1}(x)- V_{0}(x)=\textrm{Q}_k(\mathbf{X}_1).
% \end{aligned}
% \end{equation}
%The Moment \ac{DDP} approach consisting of (\ref{eq:sos_backward_2t})-(\ref{eq:sos_backward_1t}) and (\ref{eq:sdp_forward_first})-(\ref{eq:sdp_forward_2t}) solves (\ref{eq:undecomposedsdp}) in a decomposed iterative fashion. 

To facilitate these derivations, we say the moments $\hat{\mathbf{q}}_{1,z}$ computed by the SDP relaxation (\ref{eq:sdp_forward_first}) are elements of the \emph{relaxed} epigraph set, which we define as
\begin{equation}\label{eq:relaxedepigraph}
\begin{aligned}
\tilde{\mathbf{Y}}_{1,z}:=\{\mathbf{{q}}_{1,z}\in\mathbb{R}^{n_\mathbf{q}}\,\,:\,\, &M_k(\mathbf{q}_{1,z})\succeq 0;\\
&M_{k-d_{g_{j,1}}}(g_{j,1}\mathbf{q}_{1,z})\succeq0, \quad j=1,\ldots,N_{g_x};\\
&M_{k-d_{V_{1,i}}}((y-V_{1,i})\mathbf{q}_{1,z})\succeq0, \quad i={0,\ldots,z-1};\\
&M_{k-d_{y}}((\overline{y}-y)\mathbf{q}_{1,z})\succeq0\}.
\end{aligned}
\end{equation}

\begin{lemma}\label{lemma:tightening}
	  If $\theta_{LB,z} < \rho_{UB,z}$ at some iteration $z$, the relaxed epigraph set strictly tightens, i.e.~${\tilde{\mathbf{Y}}}_{1,z} \subset 	\tilde{\mathbf{Y}}_{1,z-1}$. Moreover $\theta_{LB,z+1} \geq \theta_{LB,z}$.
\end{lemma}
\begin{proof}
	Let $(\hat{\mathbf{m}}_{0,z}, \hat{\mathbf{q}}_{1,z}, \hat{\mathbf{m}}_{1,z}, \hat{\mathbf{q}}_{2,z})$ be a solution computed by the moment relaxations (\ref{eq:sdp_forward_first}) and (\ref{eq:sdp_forward_2t}) during the forward simulation. Let $\langle\hat{\mathbf{V}}_{1,z}, \hat{\mathbf{q}}_{1,z}\rangle$ be the optimal value of the backward recursion program \eqref{eq:sos_backward_2t}. By definitions of $\theta_{LB,z}$ and $\rho_{UB,z}$, and strong duality between the second-stage problems \eqref{eq:sdp_forward_2t} and \eqref{eq:sos_backward_2t}, we have
	\begin{align*}
	\theta_{LB,z}={L}_{\hat{\mathbf{m}}_{0,z}}(l_0)+L_{\hat{\mathbf{q}}_{1,z}}(y) \quad \text{and} \quad \rho_{UB,z} & ={L}_{\hat{\mathbf{m}}_{0,z}}(l_0)+L_{\hat{\mathbf{m}}_{1,z}}(l_1)+{L}_{\hat{\mathbf{q}}_{2,z}}(H) \nonumber \\
   & ={L}_{\hat{\mathbf{m}}_{0,z}}(l_0) + \langle\hat{\mathbf{V}}_{1,z}, \hat{\mathbf{q}}_{1,z}\rangle.
	\end{align*}
	Thus, $\theta_{LB,z} < \rho_{UB,z}$ implies ${L}_{\hat{\mathbf{q}}_{1,z}}(y) <  \langle\hat{\mathbf{V}}_{1,z}, \hat{\mathbf{q}}_{1,z}\rangle$. For the next iteration, we add the LMI constraint $M_{k-d_{\hat{V}_{1,z}}}((y-\hat{V}_{1,z})\mathbf{q}_{1,z})\succeq0$ to $\tilde{\mathbf{Y}}_{1,z}$, and it is straightforward to show (see \cite[eq.~(14)]{Molzahn2015} for a similar example) that the first diagonal element of this matrix is the linear expression $L_{\mathbf{q}_{1,z}}(y)- L_{\mathbf{q}_{1,z}}{(\hat{V}_{1,z})}={q}_{1,z}^{01}-\langle\hat{\mathbf{V}}_{1,z},\mathbf{q}_{1,z}\rangle$. Because this is on the diagonal of a matrix that is constrained to be positive semidefinite, it must be nonnegative. Thus the new set $\tilde{\mathbf{Y}}_{1,z}$ contains the constraint that $L_{\mathbf{q}_{1,z}}(y) \geq \langle\hat{\mathbf{V}}_{1,z}, {\mathbf{q}}_{1,z}\rangle$.
    
     The old moment vector $\hat{\mathbf{q}}_{1,z}$ is now infeasible at iteration $z+1$. Thus, $\tilde{\mathbf{Y}}_{1,z}$ must be a strict subset of $\tilde{\mathbf{Y}}_{1,z-1}$. Since \eqref{eq:sdp_forward_first} is a minimization over a subset of the previous feasible set, the cost attained may be no lower than at the previous iteration.
\end{proof}

Let $\rho^*_k$ be the optimal value of the undecomposed moment relaxation of \eqref{eq:gmp} with $T=2$:
\begin{equation}\label{eq:undecomposedsdp}
\begin{aligned}
\rho^*_k:=\min_{\mathbf{m}_{0},\mathbf{q}_{1},\mathbf{m}_{1},\mathbf{q}_{2}}\enskip &L_{\mathbf{m}_0}(l_0)+L_{\mathbf{m}_1}(l_1)+L_{\mathbf{q}_2}(H)\\
&\text{s.t.}\quad\textrm{\eqref{eq:primalsdpdyn_first}-\eqref{eq:primalsdppsd_first}},\textrm{\eqref{eq:primalsdpdyn_2t}-\eqref{eq:primalsdppsd_2t}}
\end{aligned}
\end{equation}
The following lemmas bound the possible values of the lower and upper bounds returned by Algorithm \ref{alg:moment_sos}:
\begin{lemma}\label{lemma:upperboundbound}
	At any iteration $z$, $\rho_{UB,z} \geq \rho_k^*$, the optimal value of the undecomposed moment relaxation \eqref{eq:undecomposedsdp}.
\end{lemma}
\begin{proof}
 %Let $\theta_{0,z}$ be the optimal value of the first stage \ac{SOS} \eqref{eq:sos_backward_1t}. For a polynomial $V_{1,z}(x)$, the first stage is least restricted if $V_{0,z}(x)=V_{1,z}(x)$, which is a feasible solution of \eqref{eq:putinar_S_backward_1t}\footnote{By setting the SOS coefficients $\mathbf{\sigma}_{1}$ to zero and matching the coefficients of $V_{1,z}(x)$ and $V_{0}(x,z)$ in \eqref{eq:putinar_S_backward_1t}, we obtain $V_{1,z}(x)=V_{0,z}(x)$.}. In case of a point-wise maximum of several polynomials with the same degree, we can have $\max_{i=1,\ldots,z} V_{i,z}(x) > V_{0,z}(x)$ for some $x$ in $\mathbf{X}_1$, which is more restrictive and thus leads to a lower $\theta_{0,z}$.
Let $(\hat{\mathbf{m}}_{0,z}, \hat{\mathbf{q}}_{1,z}, \hat{\mathbf{m}}_{1,z}, \hat{\mathbf{q}}_{2,z})$ be a solution computed by the moment relaxations (\ref{eq:sdp_forward_first}) and (\ref{eq:sdp_forward_2t}) during the forward simulation. Examination of the constraints of \eqref{eq:undecomposedsdp} shows that this is a feasible but in general suboptimal solution, thus $\rho_{UB,z} = {L}_{\hat{\mathbf{m}}_0}(l_0)+{L}_{\hat{\mathbf{m}}_1}(l_1)+{L}_{\hat{\mathbf{q}}_2}(H) \geq \rho_k^* $.   
\end{proof}

\begin{lemma}\label{lemma:lowerboundbound}
At any iteration $z$, $\theta_{LB,z} \leq \rho^*$, the optimal value of the \ac{GMP} \eqref{eq:gmp}.
\end{lemma}
\begin{proof}
% \item For any measure $\nu_1$, a function $V_{1,z}(x)$ generated in \eqref{eq:sos_backward_2t} satisfies $V_{1,z}(x) \leq V^*_1(x)$, where $V^*_1(x)$ is the optimal solution of \eqref{} for $t=1$, because the \ac{SOS} program \eqref{eq:sos_backward_2t} is more restrictive than \eqref{} for $t=1$.
By inserting the optimal solution $V_1^*(x)$ of the undecomposed \ac{LP} \eqref{eq:dual} with $T=2$ into the epigraph of the first stage \ac{LP} \eqref{eq:dual_1t}, it can be seen that the optimal value $\theta_t$ of \eqref{eq:dual_1t} is bounded from above by the optimal values $\theta^*=\rho^*$ of the undecomposed \acp{LP} \eqref{eq:gmp} and \eqref{eq:dual} with $T=2$. Since the \ac{SOS} approximation \eqref{eq:sos_backward_1t} of the first stage \ac{LP} \ref{eq:dual_1t} is more restricted, we have $\theta_{LB,z}\leq\rho^*$.
% than If $V_{t,i}(x) \leq V^*(t)(x)$ for all $i$ and for all $x\in\mathbf{X}_t$, then the   optimal value $\theta_{LB,z}$ of  $\max_i V_{t,i}(x) \leq V^*(t)(x)$ for all $x\in\mathbf{X}_t$. olution $V_{0,z}$ of the first stages since first stage () by monotonicity of the Bellman operator (similar argument not relying on Bellman operator definition), new VF is an underestimator of $V^*_{t-1}(x)$
\end{proof}

Finally, we can state the following result concerning the convergence of Algorithm \ref{alg:moment_sos}:
\begin{theorem}\label{th:convergence}
	Given a tolerance $\epsilon > 0$, Algorithm \ref{alg:moment_sos} attains $\rho_{UB,z}-\theta_{LB,z} \leq \epsilon$ in a finite number of iterations when applied to \ac{GMP} \eqref{eq:gmp} with $T=2$. Moreover, the sequence $\{\theta_{LB,z}\}$ converges to a value $\hat{\theta}_{LB}$ satisfying $\rho_k^* \leq \hat{\theta}_{LB} \leq \rho^*$, where $\rho^*$ is the optimal value of the original multi-stage \ac{GMP} \eqref{eq:gmp} and $\rho_k^*$ is the optimal value of its degree-$k$ moment relaxation \eqref{eq:undecomposedsdp}.
\end{theorem}
\begin{proof}
% 	Let $y^*_{\hat{q}_{1,z}}$ be the optimal value of the following optimization problem:
% 	\begin{subequations}\label{eq:relaxed_min}
% 		\begin{align}
% 		y^*_{\hat{q}_{1,z}}:=\min_{\mathbf{q}_1}\enskip &L_{\mathbf{q}_1}(y)\\
% 		\text{s.t.}\quad &\mathbf{q}_1\in\tilde{\mathbf{Y}}_{1,z+1} \label{eq:relaxed_min_yc} \\
% 		&q_1^{\alpha 0}=\hat{q}_{1,z}^{\alpha 0},\quad\alpha \in \mathbb{N}^{n_x},\sum_{i=1}^{n_x} \alpha_{i} \leq \lfloor 2k/\kappa_1\rfloor.
% 		\end{align}
% 	\end{subequations}
% 	In words, $y^*_{\hat{q}_{1,z}}$ is the relaxed epigraph value evaluated for the \emph{current} state moments $\hat{q}_{1,z}^{\alpha 0}$ with respect to the \emph{new} relaxed epigraph set $\tilde{\mathbf{Y}}_{1,z+1}$. 
% 	Since the constraint $L_{\mathbf{q}_{1,z}}(y)\geq \langle\hat{\mathbf{V}}_{1,z},{\mathbf{q}}_{1,z}\rangle$ is contained in  $\tilde{\mathbf{Y}}_{1,z+1}$, we can state that
% 	\begin{equation}\label{eq:upperboundbounded}
% 	\rho_{UB,z}= {L}_{\hat{\mathbf{m}}_{0,z}}(l_0) + \langle\hat{\mathbf{V}}_{1,z},{\hat{\mathbf{q}}}_{1,z}\rangle \leq {L}_{\hat{\mathbf{m}}_{0,z}}(l_0)+ y^*_{\hat{q}_{1,z}}.
% 	\end{equation}
	
	Let $\{\rho_{UB,z}\}$ and $\{\theta_{LB,z}\}$ be sequences over $z$ iterations. Assumption \ref{as:polynomialcompact} (continuity and compactness) implies that the sequences $\{\rho_{UB,z}\}$ and $\{\theta_{LB,z}\}$ are bounded. From Lemma \ref{lemma:tightening}, $\{\theta_{LB,z}\}$ is a monotonically increasing sequence. By the monotone convergence theorem, $\{\theta_{LB,z}\}$ converges to some accumulation point $\hat{\theta}_{LB}$. By the Bolzano-Weierstrass theorem, there is a subsequence $\{\rho_{UB,i}\}$ that converges to an accumulation point $\hat{\rho}_{UB}$. Every subsequence of a convergent sequence is also convergent, so we have $\lim_{i\rightarrow\infty}\theta_{LB,i} = \hat{\theta}_{LB}$.
    
Let $\tilde{\mathbf{X}}_{1}:=\{\mathbf{{x}}_{1}\in\mathbb{R}^{n_\mathbf{x}}\,\,:\, M_k(\mathbf{x}_1)\succeq 0;M_{k-d_{g_{j,1}}}(g_{j,1}\mathbf{x}_{1})\succeq0,j=1,\ldots,N_{g_x}\}$ be the relaxed state space, where $\mathbf{x}_{1}$ is defined in the same manner as $\mathbf{q}_1$ but without the epigraph variable $y$. For any state moment vector $\mathbf{x}_{1} \in \tilde{\mathbf{X}}_1$, a sequence $\{y^*_{\mathbf{x}_{1},z}\}$ can be constructed by solving the following optimization problem at each iteration $z$:
\begin{subequations}\label{eq:relaxed_min}
		\begin{align}
		y^*_{\mathbf{x}_{1},z}:=\min_{\mathbf{q}_1}\enskip &L_{\mathbf{q}_1}(y)\\
		\text{s.t.}\quad &\mathbf{q}_1\in\tilde{\mathbf{Y}}_{1,z} \label{eq:relaxed_min_yc} \\
		&q_1^{\alpha 0}={x}_{1}^{\alpha},\quad\alpha \in \mathbb{N}^{n_x},\sum_{i=1}^{n_x} \alpha_{i} \leq \lfloor 2k/\kappa_1\rfloor.
		\end{align}
	\end{subequations}
	In words, $y^*_{\mathbf{x}_{1},z}$ is the relaxed epigraph value evaluated for the state moments ${x}_{1}^{\alpha}$ with respect to the relaxed epigraph set $\tilde{\mathbf{Y}}_{1,z}$.
For each $\mathbf{x}_{1}$ in $\tilde{\mathbf{X}}_1$, the sequence $\{y^*_{\mathbf{x}_{1},z}\}$ is monotonically increasing (see Lemma \ref{lemma:tightening}) and bounded, and thus by the monotone convergence theorem, the limit $\{y^*_{\mathbf{x}_{1},z}\} \rightarrow y^*_{\mathbf{x}_{1},\infty}$ always exists. At no iteration $z$ of the algorithm can the backward recursion generate another $\hat{V}_{1,z}(x)$ such that $\langle\hat{\mathbf{V}}_{1,z},\mathbf{q}_{1,z}\rangle >y^*_{\mathbf{x}_{1,\infty}}$, where we choose $\mathbf{x}_{1}$ to have the same state moments as $\mathbf{q}_{1,z}$. This implies that 
\begin{equation}\label{eq:upperboundbounded}
	\lim_{i\rightarrow\infty}{L}_{\hat{\mathbf{m}}_{0,i}}(l_0)+{L}_{\hat{\mathbf{q}}_{1,i}}(y)= \lim_{i\rightarrow\infty}{L}_{\hat{\mathbf{m}}_{0,i}}(l_0)+y^*_{\hat{\mathbf{x}}_{1},i}\geq \lim_{i\rightarrow\infty}{L}_{\hat{\mathbf{m}}_{0,i}}(l_0)+\langle\hat{\mathbf{V}}_{1,i},\hat{\mathbf{q}}_{1,i}\rangle=\hat{\rho}_{UB}.
\end{equation}

	As long as $\theta_{LB,i} \leq \rho_{UB,i} - \epsilon$, that is, the termination criterion has not yet been satisfied, relation \eqref{eq:upperboundbounded} implies that the subsequence $\{\rho_{UB,i}\}$ must also converge to $\hat{\theta}_{LB}$. Thus, by virtue of Lemmas \ref{lemma:upperboundbound} and \ref{lemma:lowerboundbound}, we obtain $\rho_k^* \leq \hat{\rho}_{UB} = \hat{\theta}_{LB} \leq \rho^*$.
	
	Given $\epsilon /2 >0$, by the definition of a convergent sequence, there exists $Z\in\mathbb{N}^+$ and $I\in\mathbb{N}^+$ such that $|\theta_{LB,z}-\hat{\theta}_{LB}|<\epsilon/2\enskip \text{if}\enskip z>Z$ and $|\rho_{UB,i}-\hat{\theta}_{LB}|<\epsilon/2\enskip \text{if}\enskip i>I$. Thus there exists $J\in\mathbb{N}^+$ such that $|\rho_{UB,z}-\theta_{LB,z}|\leq |\rho_{UB,z}-\hat{\theta}_{LB}|+ |\theta_{LB,z}-\hat{\theta}_{LB}|<\epsilon\enskip \text{if}\enskip z>J$.
\end{proof}

Based on Lemma \ref{re:momentconv}, we can state that the higher the relaxation degree $2k$, the closer the undecomposed moment relaxation (\ref{eq:undecomposedsdp}) and therefore $\rho^*_k$ to the true optimal value $\rho ^*$ of the original \ac{GMP} \eqref{eq:gmp} with $T=2$, since problem (\ref{eq:undecomposedsdp}) becomes an ever tighter relaxation of \eqref{eq:gmp}.

The extension of the convergence properties to the case of multiple stages can be inferred by backward induction. If we add one new stage before the two-stage problem, the original two-stage problem (\ref{eq:undecomposedsdp}) can be seen as the nested second stage of a new upper-level two-stage problem. The nested second stage converges according to Theorem \ref{th:convergence} for given initial moments generated by the first stage. We can then apply the same arguments used for the nested problem to show the convergence of the new upper-level two-stage problem.

%% Numerical results ----------------------------
%% ----------------------------------------------
\section{Numerical results}\label{numerics}

We evaluate the algorithm using a real-world long-term borehole storage problem. The Moment \ac{DDP} approach developed in Section \ref{MomentSOS} is compared with the \ac{DP} approach using discretization of the state/action space for the case of a small storage system in Section \ref{smallexample}. The convergence of the algorithm for a larger problem with multiple storage systems is then shown in Section \ref{largesystem}.
\subsection{Single storage system}\label{smallexample}
We consider the system pictured in Fig.~\ref{fig:setup}, similar to the setup in \cite{DeRidder2011}, comprising a borehole, a \ac{HP}, a chiller and a boiler. The objective is to satisfy the heating and cooling demand, which vary by time of year, at minimum annual cost. Heating can be supplied either by the boiler or by the \ac{HP} that draws energy from the borehole. The efficiency of the \ac{HP} depends on the outlet temperature of the borehole. The cooling demand can be satisfied by either running the chiller or by charging the borehole through a heat-exchanger. 

\begin{figure}
	\centering
	\includegraphics[scale=.25]{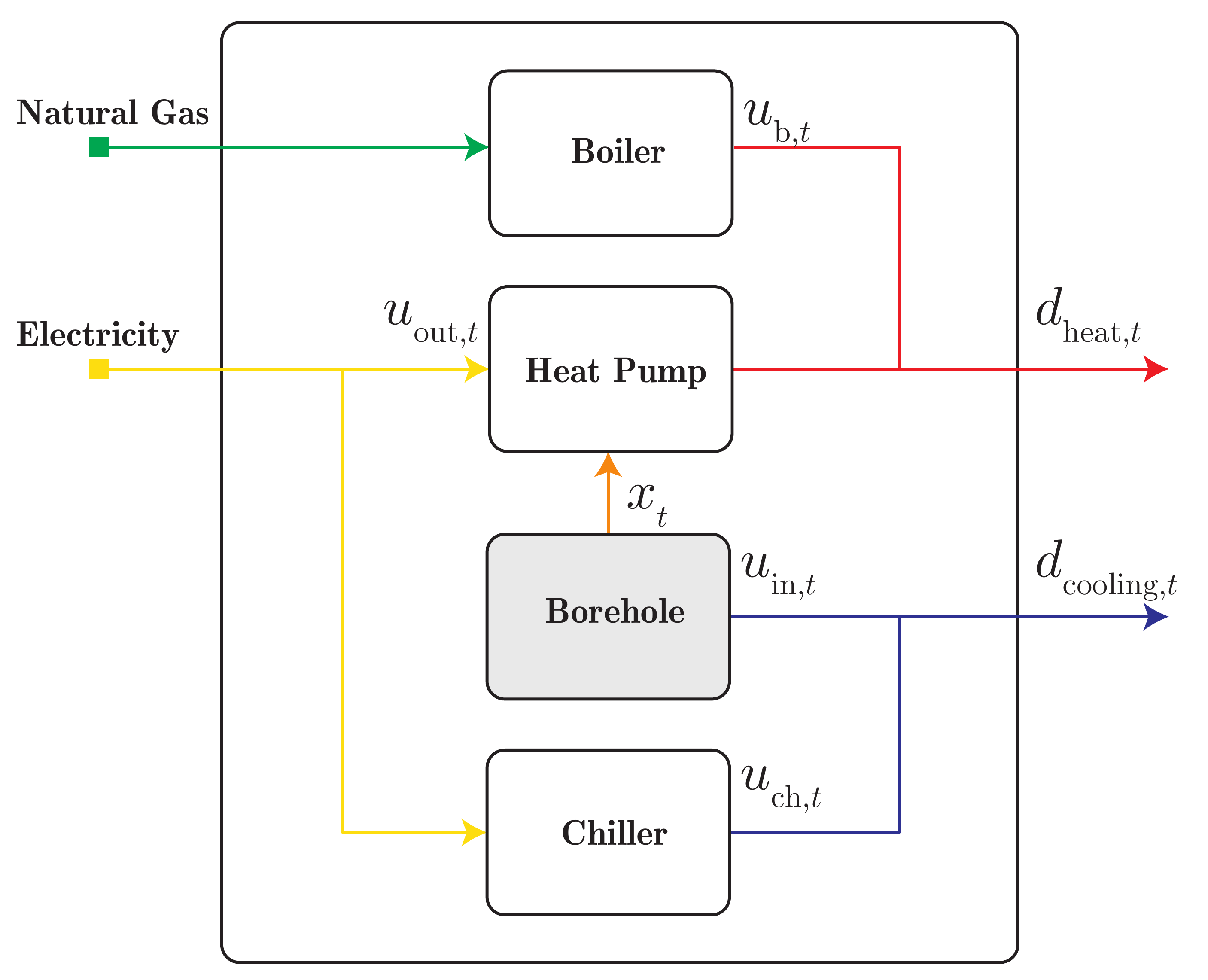}
	\caption[Schematic of the energy system with borehole storage]{Schematic of the energy system with borehole storage}\label{fig:setup}
\end{figure}

This system is sufficiently small for the \ac{DP} approach using discretization to be tractable. We evaluate the quality of the approximate value function generated by the Moment \ac{DDP} approach, as well as the quality of the solution when the approximate value functions are used in a single-stage optimal control problem in comparison with the discretized \ac{DP} solution. We assume the heating and cooling demand to be given and use measurements from the Empa Campus in D\"ubendorf Switzerland (Fig.~\ref{fig:heatingdemand}\if\thesismode0 in the Appendix\fi) scaled for a single storage application. The characteristics of the ground borehole are derived from a thermal response test conducted on the Empa campus. The long-term \ac{ESMP} over the horizon of one year is formulated as follows:
\begin{subequations}\label{eq:boreholenonconvex}
	\begin{align}
	&\min_{\{x_{t}\}_{t=1}^T,\{u_{\textrm{in},t},u_{\textrm{out},t}, u_{\textrm{b},t},u_{\textrm{ch},t}\}_{t=0}^{T-1}}\sum_{t=0}^{T-1} c_e (u_{\textrm{out},t}+u_{\textrm{ch},t})+c_\textrm{g} u_{\textrm{b},t}\\
	&\text{s.t.}\enskip\; x_{t+1}=x_t+\Delta t \frac{1}{mc}(\lambda(x_t-T_{\infty})-a(x_t) u_{\textrm{out},t}+u_{\textrm{in},t}),\quad t=0,\ldots,T-1,\label{eq:boreholedynamic}\\
	&\quad\quad a(x_t) u_{\textrm{out},t} + a_\textrm{b} u_{\textrm{b},t} = d_{\textrm{heat},t},\quad t=0,\ldots,T-1,\label{eq:heating_balance}\\
	&\quad\quad u_{\textrm{in},t}+a_{\textrm{ch}} u_{\textrm{ch},t} = d_{\textrm{cooling},t},\quad t=0,\ldots,T-1,\label{eq:cooling_balance}\\
	&\quad\quad\underline{T} \leq x_t \leq \overline{T}, \quad t=1,\ldots,T,\\
	&\quad\quad 0 \leq u_{\textrm{out},t} \leq \overline{u}_{\textrm{out}};\enskip 0 \leq u_{\textrm{in},t} \leq \overline{u}_{in},0 \leq u_{\textrm{b},t}  \leq \overline{u}_{\textrm{b}};\enskip 0 \leq u_{\textrm{ch},t}  \leq \overline{u}_{\textrm{ch}},\quad t=0,\ldots,T-1,
	\end{align}
\end{subequations}
where $x_t$ is the ground temperature, $u_{\textrm{in},t}$ the storage charge, $u_{\textrm{out},t}$ the \ac{HP} power when drawing energy from the ground, $u_{\textrm{ch},t}$ the chiller power and $u_{\textrm{b},t}$ the boiler power. The heating and cooling demands are denoted as $d_{\textrm{heat},t}$ and $d_{\textrm{cooling},t}$. The power rating limits are denoted by $\overline{u}_{\textrm{out}}$, $\overline{u}_{in}$, $\overline{u}_{\textrm{ch}}$ and $\overline{u}_{\textrm{b}}$. The temperature of the borehole $x_t$ is specified to remain within $[\underline{T},\overline{T}]$. $T_{\infty}$ denotes the boundary ground temperature, $\lambda$ the thermal conductivity  and $mc$ the thermal inertia of the ground. If ground temperatures are not available for measurement, the model provided in \cite{atam2016} can be used instead. We set $T=12$ to obtain monthly value functions, leading to $\Delta t=730$ hours for (\ref{eq:boreholedynamic}). A linear function $a(x_t)$ was fitted to the measurements of the \ac{COP} of the \ac{HP} in the Energy Hub of the NEST building on the Empa Campus (see Fig.~\ref{fig:cop} \if\thesismode0in the Appendix \fi). The third column of Table~\ref{table:inputdata} \if\thesismode0in the Appendix \fi summarizes all the numerical energy system data for (\ref{eq:boreholenonconvex}). Due to the temperature-dependent \ac{COP}, the storage problem (\ref{eq:boreholenonconvex}) is non-convex. After eliminating decision variables $u_{\textrm{b},t}$ and $u_{\textrm{ch},t}$ using the equality constraints (\ref{eq:heating_balance}) and (\ref{eq:cooling_balance}), the problem has one state $x_t$ and two control input decision variables $u_{\textrm{in},t}$ and $u_{\textrm{out},t}$.
%\begin{minipage}{0.5\textwidth}

\if\thesismode1
\begin{figure}
	  \centering
      \includegraphics[scale=1]{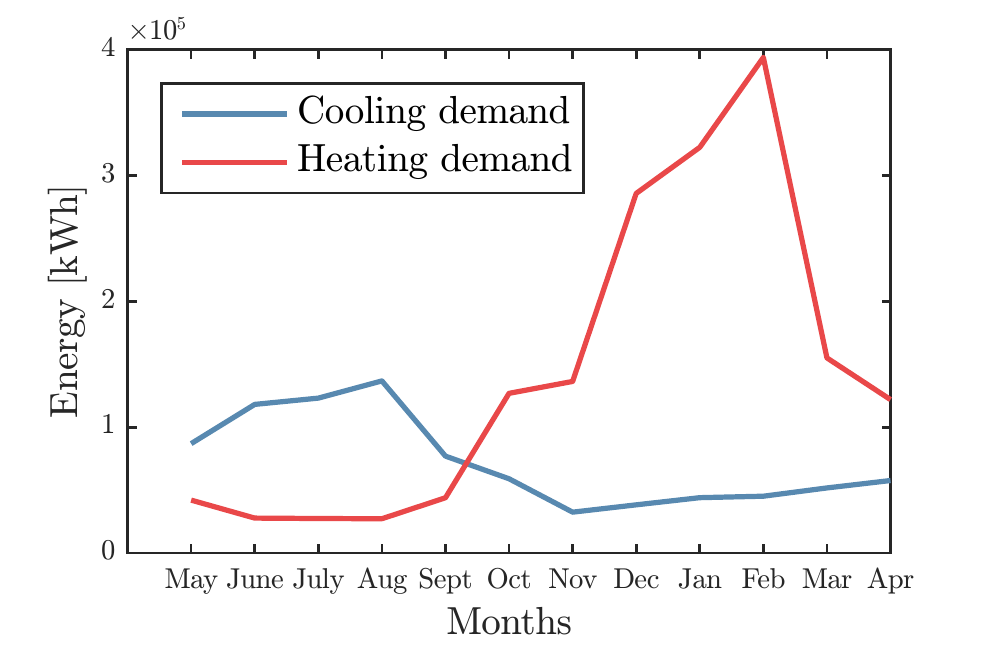}
      \caption[Heating and cooling demand of the Empa Campus]{Heating and cooling demand $d_{\textrm{heat},t}$ and $d_{\textrm{cooling},t}$ of the single storage application over a year}\label{fig:heatingdemand}
\end{figure}
\begin{figure}
	\centering
	\includegraphics[scale=1]{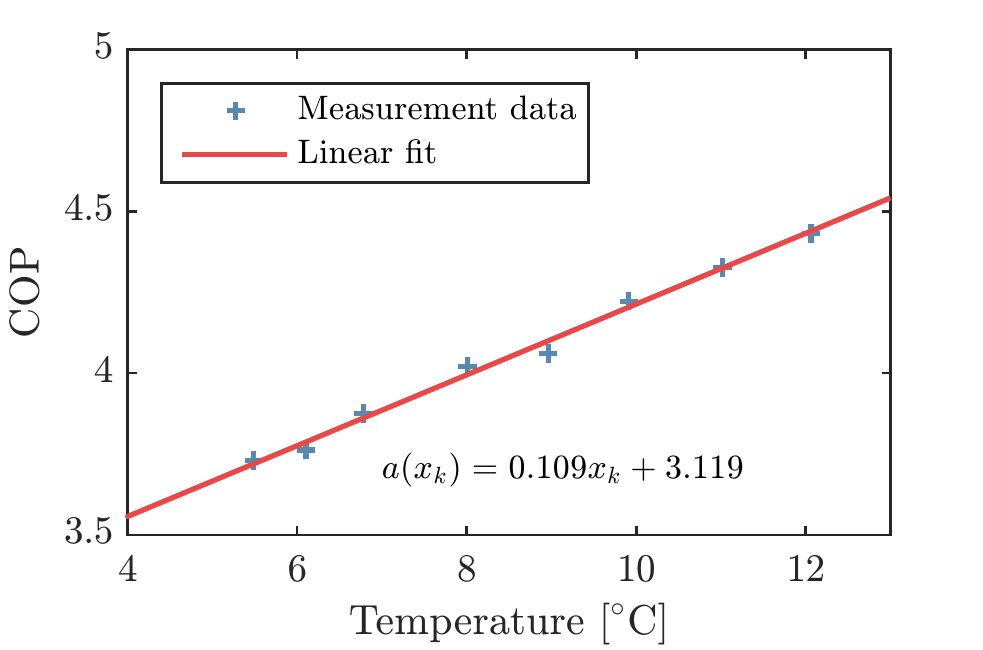}
	\caption[Fitting of the temperature-dependent COP]{Fitting of the inlet temperature-dependent \ac{COP} $a(x_t)$ of the \ac{HP} }\label{fig:cop}
\end{figure}
\begin{table}
	\caption[Energy system data]{Energy system data}
	\label{table:inputdata}
	\begin{center}
		\begin{tabular}{ l l l l}
			& Parameter &Single storage&Multiple storage\\ 
			\midrule
			\textbf{Grid Feeders} \\ Power & Cost $c_e$:& 0.096\$/kWh &0.096\$/kWh\\
			\midrule
			Gas & Cost $c_g$:& 0.063\$/kWh &0.063\$/kWh\\
			\midrule
			\textbf{Conversion} \\ \ac{HP}s& \ac{COP} $a(x_t)$:& see Fig.~\ref{fig:cop}&see Fig.~\ref{fig:cop}\\& Capacity $\overline{u}_{\rm out}$:& 60kW & 60kW \\
			\midrule
			Boiler& Efficiency $a_{\textrm{b}}$:& 0.7&0.7\\ & Capacity $\overline{u}_{b}$:& 285 kW & 855 kW\\
			\midrule
			Chiller& \ac{COP} $a_{\rm ch}$:& 5&5\\
			&  Capacity $\overline{u}_{\rm ch}$: &150kW&450kW\\
			\midrule
			\textbf{Storage} \\
			Boreholes&Conductivity $\lambda$: & 0.621kW/$^\circ$C&0.621kW/$^\circ$C$\pm 10\%$\\
			& Inertia $mc$: &14805kWh/$^\circ$C&14805kWh/$^\circ$C\\
			& Capacity $\overline{u}_{in}$:&100kW&100kW\\
			& Ground $T_\infty$:&12$^\circ$C &12$^\circ$C\\
			& Range $[\underline{T},\overline{T}]$: &[0,12]$^\circ$C&[0,12]$^\circ$C\\
			\bottomrule
		\end{tabular}
	\end{center}
\end{table}
\fi
The following value function approximations are considered to solve (\ref{eq:boreholenonconvex}):
\begin{itemize}
	\item Discretized dynamic programming with $41$ state grid points on $[\underline{T},\overline{T}]$ and $1001$ grid points  per control input on $[0,\overline{u}]$.
	\item  Moment \ac{DDP} with relaxation degree $2k=2$; this restricts the value function approximation to affine functions. (Recall that the maximum degree of the polynomial approximation of the value function is constrained by ${\rm deg}(V_{t,z})\kappa_t \leq 2k$, where in this case the highest polynomial degree found in the dynamics is $\kappa_t=2$.)
	\item  Moment \ac{DDP} with relaxation degree $2k=4$; this permits quadratic value function approximations, however in this case we add constraints to restrict all quadratic terms to zero. As a result, only affine function approximations are used.\footnote{For consistency, the primal problem over moments also has to be modified (relaxed) by removing some linear equality constraints on higher-order moments arising from the dynamics \eqref{eq:primalsdpdyn_f}. For brevity we do not detail this procedure here.}
	\item  Moment \ac{DDP} with relaxation degree $2k=4$; using the full quadratic value function approximations permitted by this relaxation degree.
\end{itemize}
The Moment \ac{DDP} approach is implemented using YALMIP \citep{Lofberg2004} and solved with MOSEK\textsuperscript{TM}. The discretized \ac{DP} problem is implemented and solved using the \emph{dpm} toolbox of \cite{Sundstrom2009} and MATLAB\textsuperscript{TM}. The problem data are scaled to be contained in the unit box to improve the numerical performance of the Moment \ac{DDP} approach.

First, we compare the accuracy of different value function bases for a uniform initial state distribution. \if\thesismode1 In Figs.~\ref{fig:valueapprox1} (May to October) and \ref{fig:valueapprox2} (November to April) \else In Fig.~\ref{fig:valueapprox}\fi, the approximate value functions are shown together with the reference computed by discretized \ac{DP}. The kinks in the \ac{DP} value functions for the months of May to August are caused by the additional cost incurred by using the chiller if the storage temperature is too high for cooling. The kinks in March and April are due to two different operating modes: using the \ac{HP} to provide heat or both, the \ac{HP} and the boiler. Affine and quadratic approximate value functions generated by relaxation $2k=4$ are a close fit for most months. For the months May to September, the lower sections of the approximate value functions are less accurate. Whereas the slopes of the approximate functions are very close to discretized \ac{DP} reference, the kink positions are not. However, as subsequent results on the performance of the resulting control policy demonstrate, using the borehole to provide cooling is still optimal. There is a considerable difference between the discretized \ac{DP} and the piecewise affine value function generated by the relaxation of order $2k=2$.
\if\thesismode1
\begin{figure}
	\centering
	\includegraphics[trim={0 1cm 0 0},clip,scale=1]{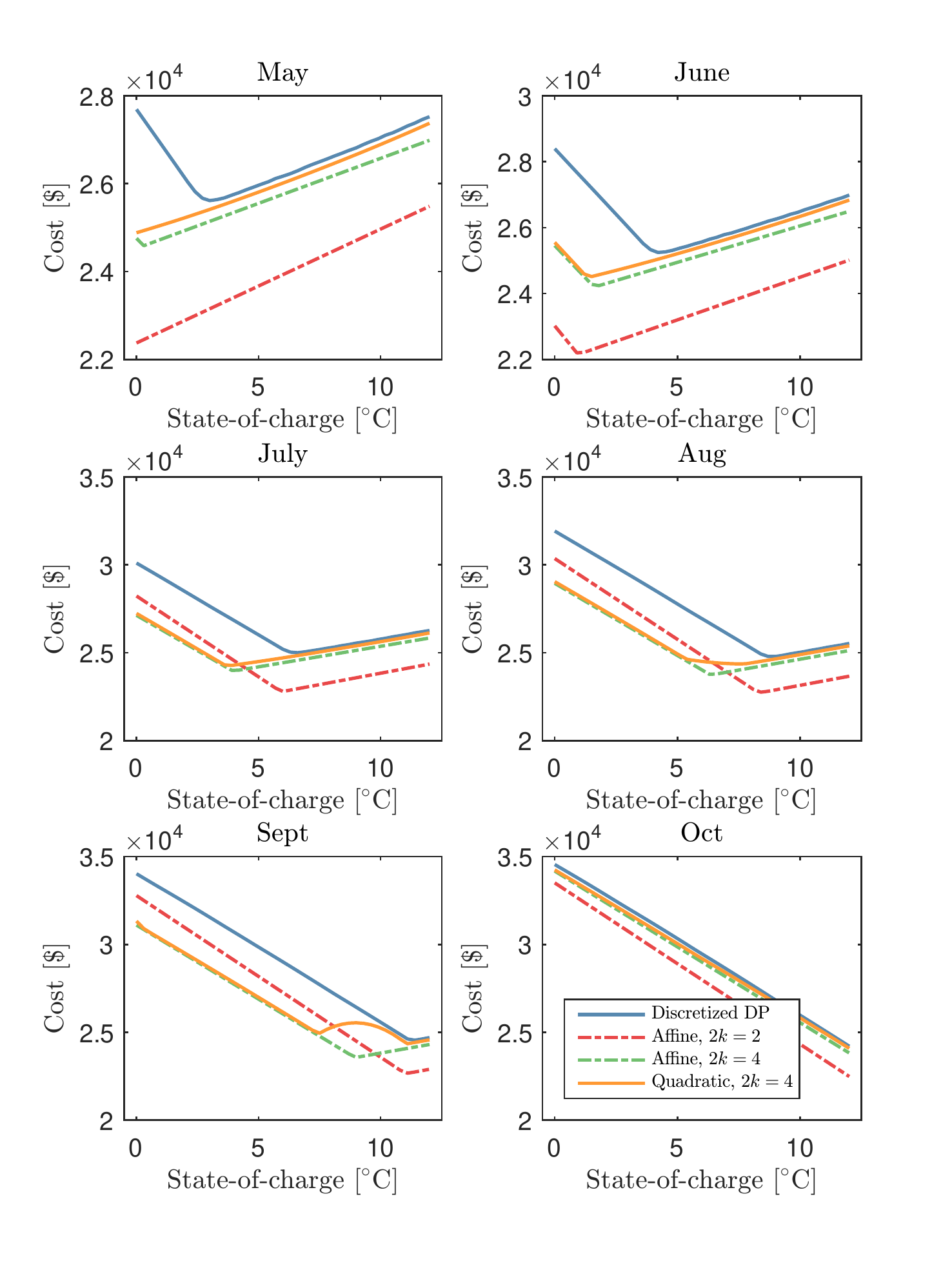}
	\caption[Value function approximations for the months of May to October]{Value function approximations for the months of May to October using different basis functions in comparison to discretized \ac{DP}}\label{fig:valueapprox1}
\end{figure}
\begin{figure}
	\centering
	\includegraphics[trim={0 1cm 0 0},clip,scale=1]{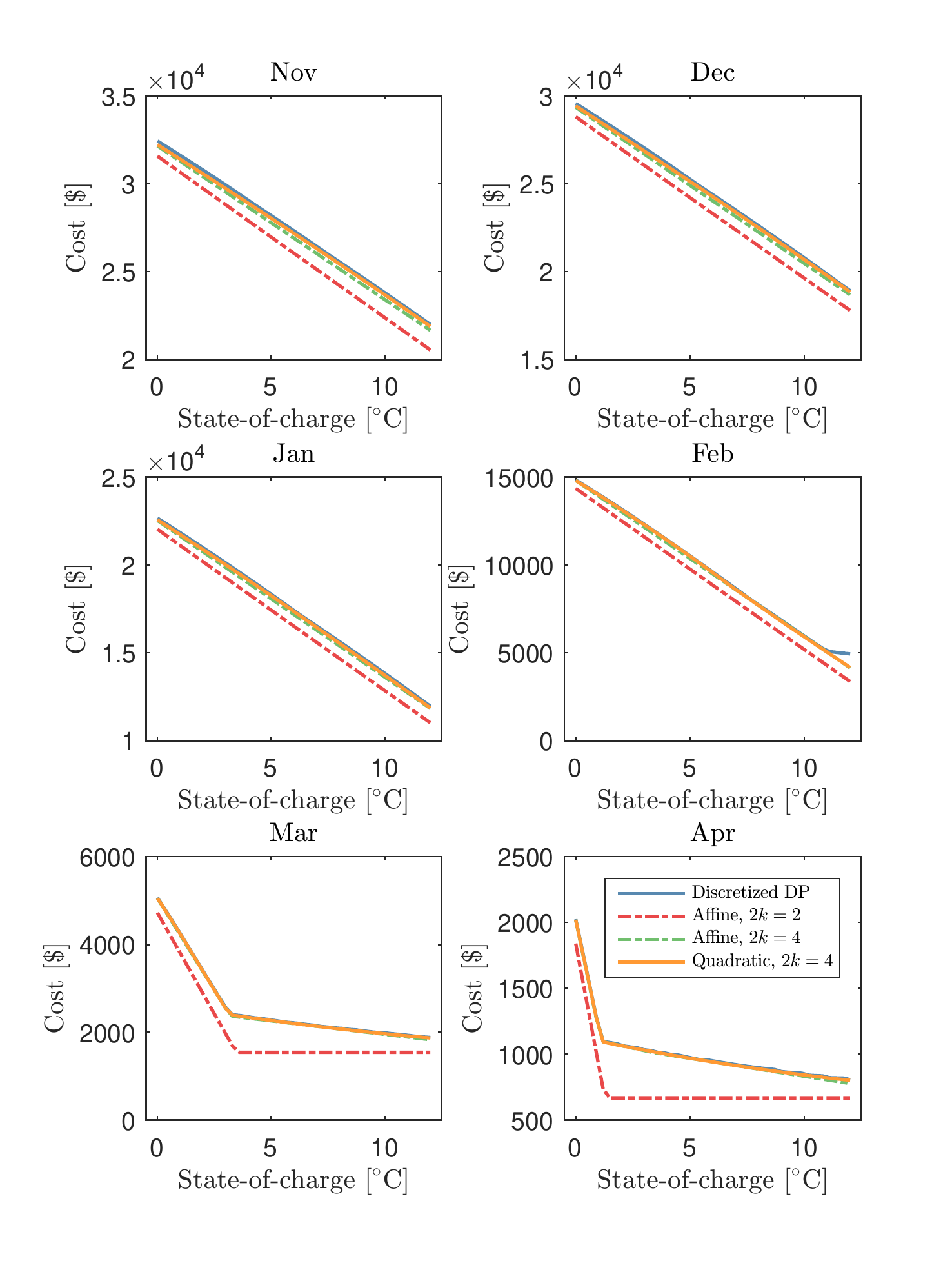}
	\caption[Value function approximations for the months of November to April]{Value function approximations for the months of November to April using different basis functions in comparison to discretized \ac{DP}}\label{fig:valueapprox2}
\end{figure}
\else
\begin{figure}
	\centering
	\includegraphics[trim={2cm 1.5cm 0cm 1cm},clip,scale=1]{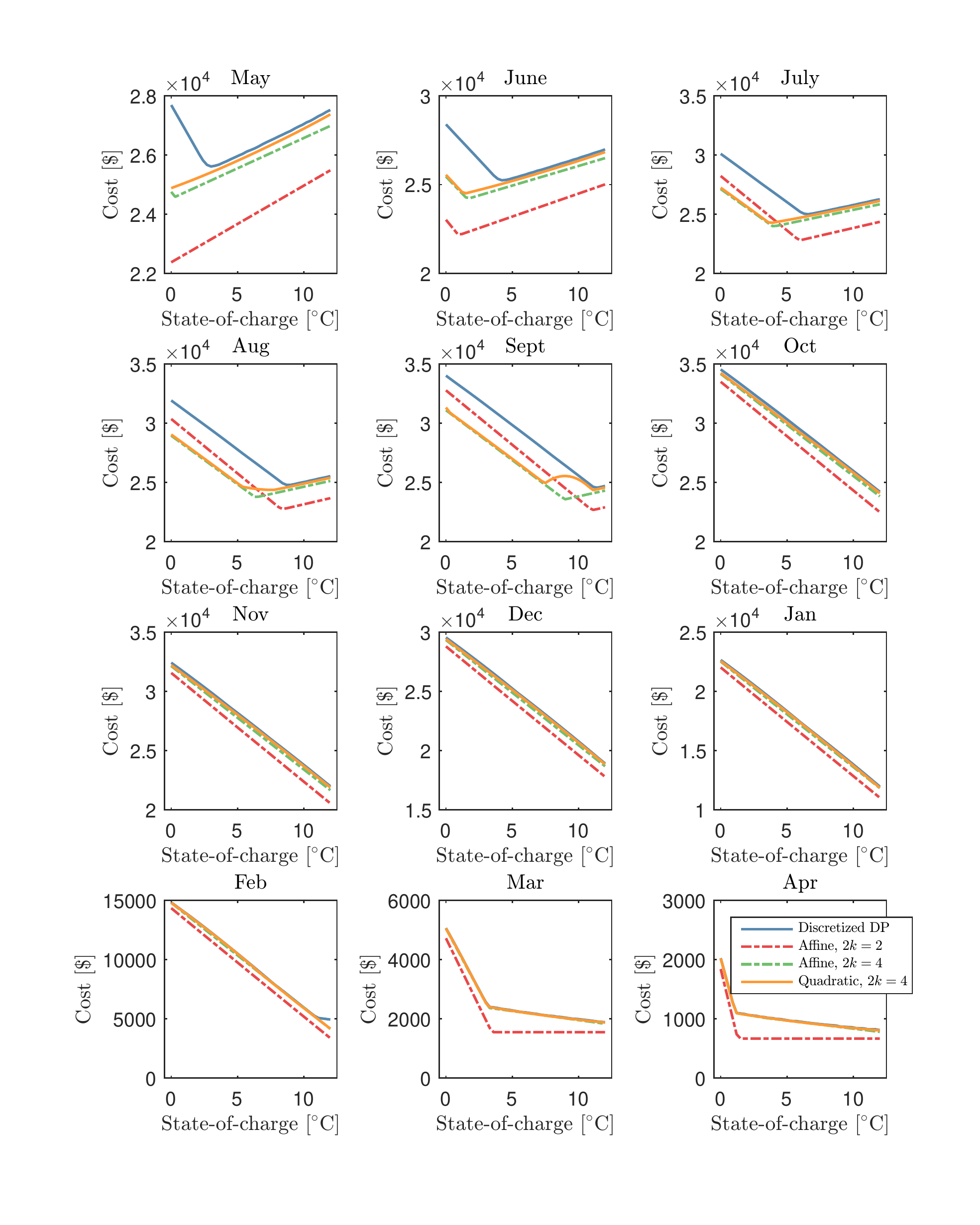}
	\caption{Value function approximations using different basis functions in comparison to discretized \ac{DP}.}\label{fig:valueapprox}
    \end{figure}
\fi

The convergence of the lower bound $\rho_{LB}$ and the upper bound $\rho_{UB}$ of the Moment \ac{DDP} algorithm for different polynomial basis functions is shown in Fig.~\ref{fig:convergence}. Affine basis functions make the Moment \ac{DDP} algorithm converge faster than quadratic basis functions. The total solver times for a predefined convergence tolerance are reported in Table~\ref{table:times} \if\thesismode0in the Appendix \fi. Note that as in conventional \ac{DDP}, problems (\ref{eq:sos_backward}) and (\ref{eq:sdp_forward}) increase slightly in size at every iteration as we add additional under-approximating value functions.
\if\thesismode1
\begin{figure}
	\centering
	\includegraphics[trim={0 1cm 0 0},clip,scale=1]{convergence}
	\caption[Single storage system: Convergence of the lower bound and the upper bound]{Single storage system: Convergence of the lower bound $\rho_{LB}$ and the upper bound $\rho_{UB}$ of the Moment \ac{DDP} algorithm for different basis functions}\label{fig:convergence}
\end{figure}
\begin{table} 
	\caption[Accumulated solver time over all iterations of the Moment DPP approach]{Accumulated MOSEK\textsuperscript{TM} solver time over all iterations of the Moment DDP approach obtained on a PC with an Intel-i5 2.2GHz CPU with 8GB RAM for a tolerance of $\epsilon=10^{-4}$ (after scaling the problem data to the unit box)}\label{table:times}
	\begin{center}
		\begin{tabular}{ l l l}
			Basis functions/Relaxation&Single storage&Multiple storage\\ 
			\midrule
			Affine value functions, $2k=2$&4.77s&6.39s\\
			Affine value functions, $2k=4$&5.77s&18.65min\\
			Quadratic value functions, $2k=4$&25.23s&28.24min\\
			\midrule
		\end{tabular}
	\end{center}
\end{table}
\else
\begin{figure}
\centering
\begin{minipage}[c]{0.45\textwidth}
	\centering
	\includegraphics[trim={0 1cm 0 0},clip,scale=1]{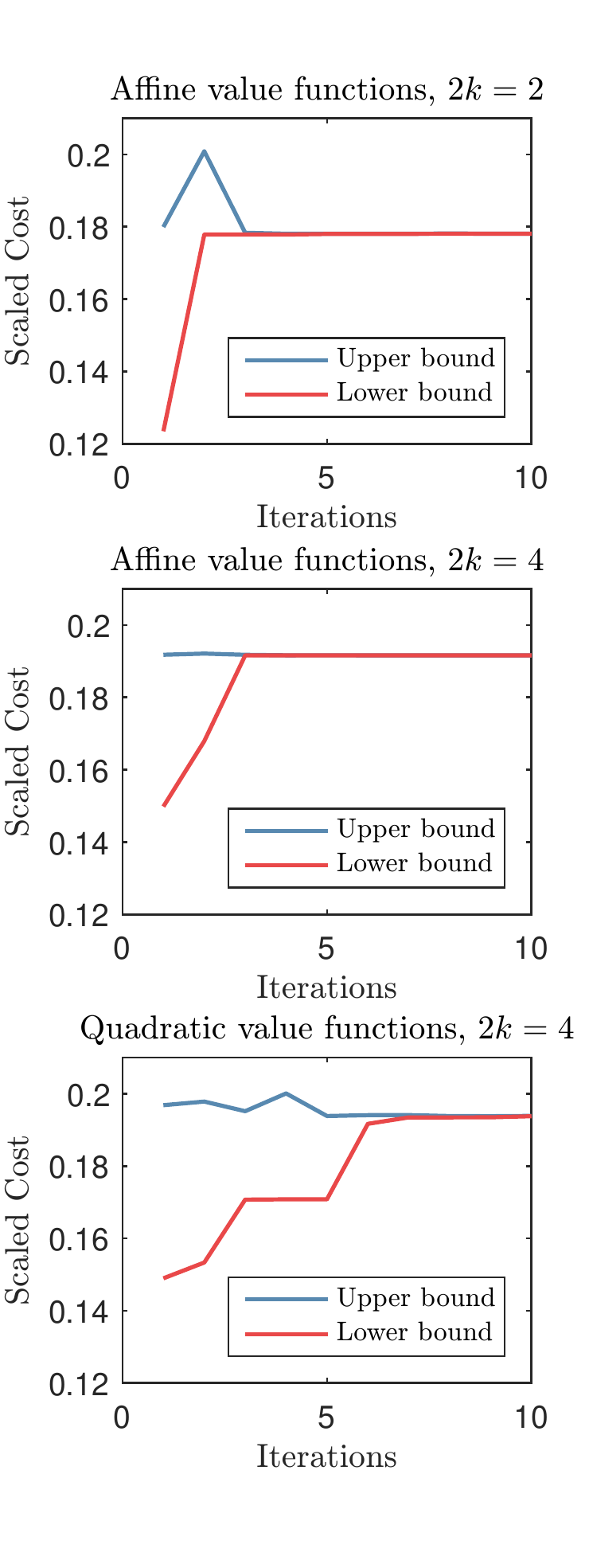}
	\caption{Single storage system: Convergence of the lower bound $\rho_{LB}$ and the upper bound $\rho_{UB}$ of the Moment \ac{DDP} algorithm for different basis functions}\label{fig:convergence}
\end{minipage}
\hspace{0.5cm}
\begin{minipage}[c]{0.45\textwidth}
	\centering
	\includegraphics[trim={.8cm 1cm 0 0},clip,scale=1]{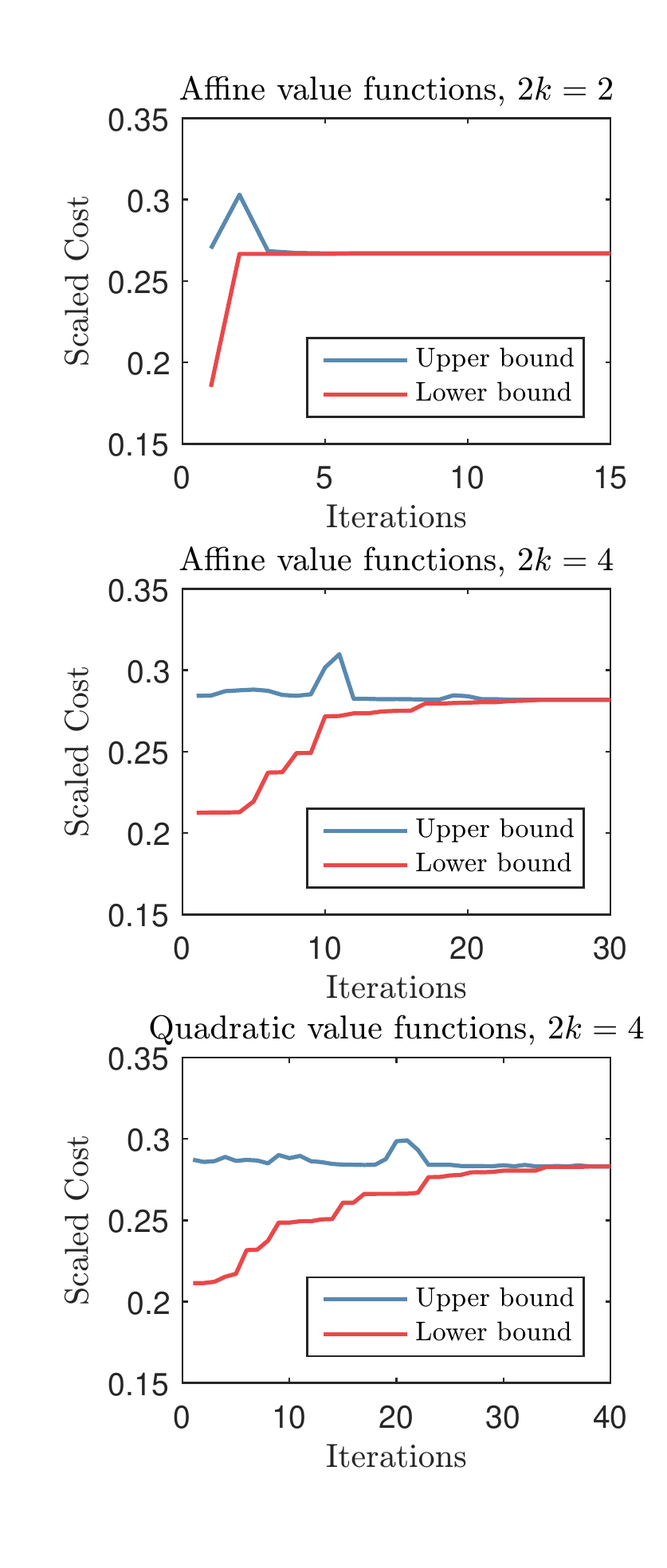}
	\caption{Multiple storage systems: Convergence of the lower bound $\rho_{LB}$ and the upper bound $\rho_{UB}$ of the  Moment \ac{DDP} algorithm for different basis functions}\label{fig:convergence3d}
\end{minipage}
\end{figure}
\fi

Finally, instead of (\ref{eq:boreholenonconvex}), we solve a sequence of single-stage problems augmented with approximate value functions obtained by the Moment \ac{DDP} approach. For each month $t\in\{1,\ldots,12\}$, we solve:
\begin{subequations}\label{eq:boreholenonconvex_shortterm}
	\begin{align}
	&\min_{x_{t+1}, u_{\textrm{in},t},u_{\textrm{out},t}, u_{\textrm{b},t},u_{\textrm{ch},t}} c_e (u_{\textrm{out},t}+u_{\textrm{ch},t})+c_\textrm{g} u_{\textrm{b},t}\nonumber\\
	&\quad\quad\quad\quad\quad\quad\quad\quad\quad\quad+\max\{V_{t+1,0}(x_{t+1}),\ldots,V_{t+1,z}(x_{t+1})\}\\
	&\text{s.t.}\enskip\; x_{t+1}=x_t+\Delta t \frac{1}{mc}(\lambda(x_t-T_{\infty})-a(x_t) u_{\textrm{out},t}+u_{\textrm{in},t})\\
	&\quad\quad a(x_t) u_{\textrm{out},t} + a_\textrm{b} u_{\textrm{b},t} = d_{\textrm{heat},t},\\
	&\quad\quad u_{\textrm{in},t}+a_{\textrm{ch}} u_{\textrm{ch},t} = d_{\textrm{cooling},t},\\
	&\quad\quad\underline{T} \leq x_{t+1} \leq \overline{T},\\
	&\quad\quad 0 \leq u_{\textrm{out},t} \leq \overline{u}_{\textrm{out}};\\
	&\quad\quad 0 \leq u_{\textrm{in},t} \leq \overline{u}_{in};\enskip 0 \leq u_{\textrm{b},t}  \leq \overline{u}_{\textrm{b}};\enskip 0 \leq u_{\textrm{ch},t}  \leq \overline{u}_{\textrm{ch}}
	\end{align}
\end{subequations}

The start of the storage cycle is assumed to be the beginning of May because the cooling overcomes the heating demand during this period (see Fig.~\ref{fig:heatingdemand}). In Fig.~\ref{fig:shortterm}, we show the total cost of operating the system over the full horizon for a uniformly distributed number of initial states when each month is solved as a single-stage problem (\ref{eq:boreholenonconvex_shortterm}). We use the generic nonlinear solver IPOPT \citep{Wachter2006} to compute a locally-optimal solution. The affine value functions perform almost as well as the forward simulation of discretized \ac{DP}. The quadratic value functions lead to sub-optimal results with a local optimization algorithm for some initial states. This might be due to the non-convexity of the approximate value function in September (see Fig.~\if\thesismode1\ref{fig:valueapprox1}\else\ref{fig:valueapprox}\fi).
\begin{figure}
	\centering
	\includegraphics[scale=1]{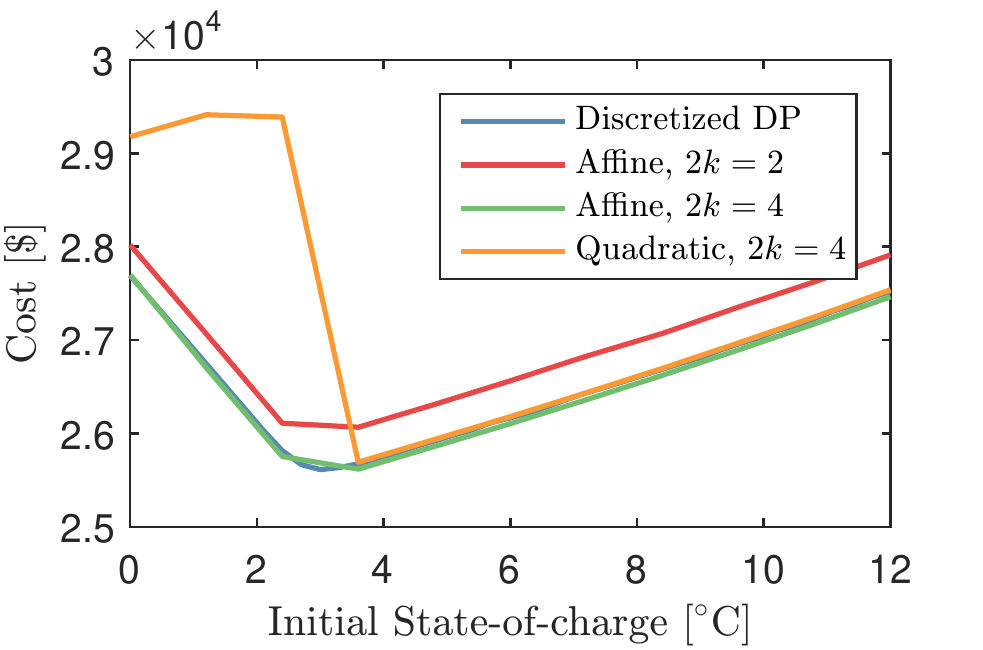}
	\caption[Cost over the full year starting in May for a sequence of single-stage problems]{Cost over the full year starting in May for a sequence of single-stage problems with different approximate functions in comparison to discretized \ac{DP}}\label{fig:shortterm}
\end{figure}
\subsection{Multiple storage systems}\label{largesystem}
We now evaluate the convergence of the Moment \ac{DDP} approach for a higher dimensional problem, namely an \ac{ESMP} with three different storage systems:
\begin{subequations}\label{eq:boreholenonconvex_large}
	\begin{align}
	&\min_{\{\{x_{i,t}\}_{t=1}^T,\{u_{\textrm{in},i,t},u_{\textrm{out},i,t}\}_{t=0}^{T-1}\}_{i=1}^3, \{u_{\textrm{b},t},u_{\textrm{ch},t}\}_{t=0}^{T-1}} \sum_{t=0}^{T-1} \Bigg(c_e u_{\textrm{ch},t}+c_g u_{\textrm{b},t}+\sum_{i=1}^3 c_e u_{\textrm{out},i,t}\Bigg)\\
	&\text{s.t.}\enskip\; x_{i,t+1}=x_{i,t}+\Delta t \frac{1}{mc}(\lambda_i(x_{i,t}-T_{\infty})+u_{\textrm{in},i,t}-a(x_{i,t}) u_{\textrm{out},i,t}),\nonumber\\
	&\quad\quad\quad\quad t=0,\ldots,T-1,i=1,2,3,\\
	&\quad\quad\sum_{i=1}^{3}a(x_{i,t}) u_{\textrm{out},i,t} + a_\textrm{b} u_{\textrm{b},t} = d_{\textrm{heat},t},
	\quad t=0,\ldots,T-1,\label{eq:heating_balance_l}\\
	&\quad\quad\sum_{i=1}^{3}u_{\text{in},i,t}+a_{\textrm{ch}} u_{\textrm{ch},t} = d_{\textrm{cooling},t}, \quad t=0,\ldots,T-1,\label{eq:cooling_balance_l}\\
	&\quad\quad\underline{T} \leq x_{i,t} \leq \overline{T}, \quad t=1,\ldots,T,i=1,2,3,\\
	&\quad\quad 0 \leq u_{\textrm{out},i,t} \leq \overline{u}_{\textrm{out}};\enskip 0 \leq u_{\textrm{in},i,t} \leq \overline{u}_{\textrm{in}},\medspace t=0,\ldots,T-1,i=1,2,3,\\
	&\quad\quad 0 \leq u_{\textrm{b},t}  \leq \overline{u}_{\textrm{b}};\enskip 0 \leq u_{\textrm{ch},t}  \leq \overline{u}_{\textrm{ch}}, \quad t=0,\ldots,T-1,
	\end{align}
\end{subequations}
With two additional boreholes, the discretized \ac{DP} approach memory requirements become excessive, since a grid must be spanned over a 9-dimensional decision space after elimination of the boiler and chiller variables using (\ref{eq:heating_balance_l}) and (\ref{eq:cooling_balance_l}). In addition to the energy system data of the fourth column of Table~\ref{table:inputdata}, we use the heating and cooling demand of the single storage example of the previous section multiplied by a factor 3 as input data. The convergence of affine and quadratic approximate value functions for a uniform initial state distribution is shown in Fig.~\ref{fig:convergence3d}. All methods converge in a reasonable number of iterations. Table~\ref{table:times} \if\thesismode0in the Appendix \fi reports the total solver times for a predefined tolerance.
\if\thesismode1
\begin{figure}
	\centering
	\includegraphics[trim={0 1cm 0 0},clip,scale=1]{convergence3d}
	\caption[Multiple storage systems: Convergence of the lower bound and the upper bound]{Multiple storage systems: Convergence of the lower bound $\rho_{LB}$ and the upper bound $\rho_{UB}$ of the  Moment \ac{DDP} algorithm for different basis functions}\label{fig:convergence3d}
\end{figure}
\fi

\section{Conclusion and Future Work}\label{conclusion}
This paper presented a novel value function approximation scheme for nonlinear multi-stage problems that leverages sum-of-squares techniques within a \ac{DDP} framework. The scheme is based on a finite-horizon \ac{GMP} for discrete-time dynamical systems. The primal, a moment problem, and the dual, an \ac{SOS} program, are used iteratively to refine the statistics of the forward state trajectory and the approximate value functions respectively. Whereas \ac{DDP} returns value functions that apply locally around trajectories emanating from a single initial state, and generally only for linear system dynamics and cost, the Moment \ac{DDP} approach returns approximate value functions for a distribution of initial states, and moreover achieves this for systems with polynomial dynamics, costs, and constraints. Depending on the degree of polynomials used, the optimal policy obtained by short-term problems augmented with approximate value functions returned by the Moment \ac{DDP} approach can be almost as cost-effective as that obtained by discretized \ac{DP}. We also demonstrated convergence of the Moment \ac{DDP} approach for a case that is computationally too demanding for discretized \ac{DP}. 

The computational complexity of the Moment \ac{DDP} approach could be reduced by exploiting any sparsity present in the problem data in \eqref{eq:energystorage} \citep{Waki2006}. This would draw on the experience of \cite{Molzahn2015} and \cite{Ghaddar2014}, who successfully exploited the sparse structure of electrical networks to obtain global solutions to the nonlinear optimal power flow problem using moment relaxations. Alternative positivity certificates, such as the one proposed in \cite{Ahmadi2014a}, also offer the possibility of reduced computational complexity.

\section*{Acknowledgments}
We would like to thank Viktor Dorer, Roy Smith and Jan Carmeliet for their valuable help and support. We are also grateful to Xinyue Li for her work on the heat pump characterization and to Paul Beuchat, Georgios Darivianakis, Benjamin Flamm, Mohammad Khosravi, and Annika Eichler for fruitful discussions.
This research project is financially supported by the Swiss Innovation Agency Innosuisse and by NanoTera.ch under the project HeatReserves, and is part of the Swiss Competence Center for Energy Research SCCER FEEB\&D.

%\end{linenumbers}
%\section*{References}
{%\small
\begingroup
    \setstretch{0.7}
   	\bibliography{ThesisReferences.bib}
\endgroup

}
%\appendix
\newpage
\section*{Appendix}
\begin{figure}[H]
	\centering
	\includegraphics[trim={0 1cm 0 0},clip,scale=1]{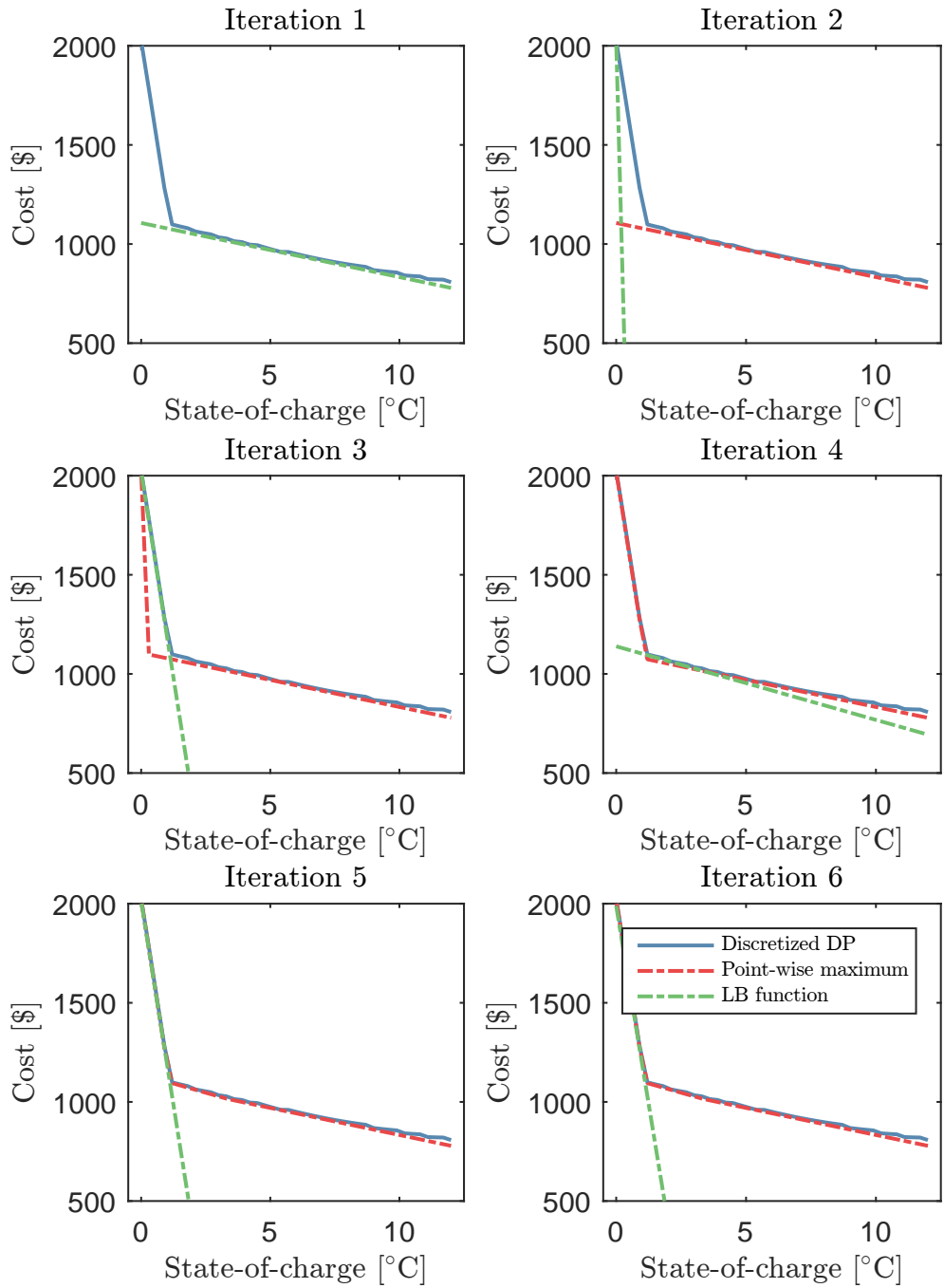}
	\caption[Sequence of six backward recursions for the single storage example]{Single storage example of Section \ref{numerics}: Cost of stored energy in the beginning of April ($t=11$) approximated using affine basis functions and $2k=4$, shown for six DDP iterations. New lower-bounding functions (LB function) are shown in green. The point-wise maximum of all previous lower-bounding functions is shown in red.}\label{fig:proofillustration}
\end{figure}
\begin{figure}[H]
	  \centering
      \includegraphics[scale=1]{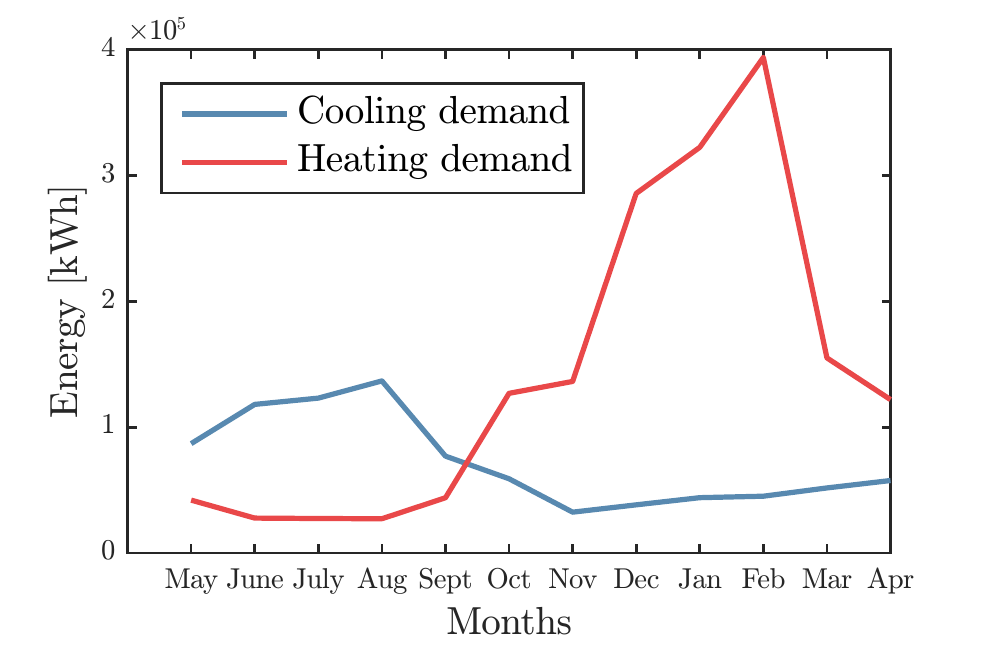}
      \caption[Heating and cooling demand of the Empa Campus]{Heating and cooling demand $d_{\textrm{heat},t}$ and $d_{\textrm{cooling},t}$ of the single storage application over a year}\label{fig:heatingdemand}
\end{figure}

\begin{figure}[H]
	\centering
	\includegraphics[scale=1]{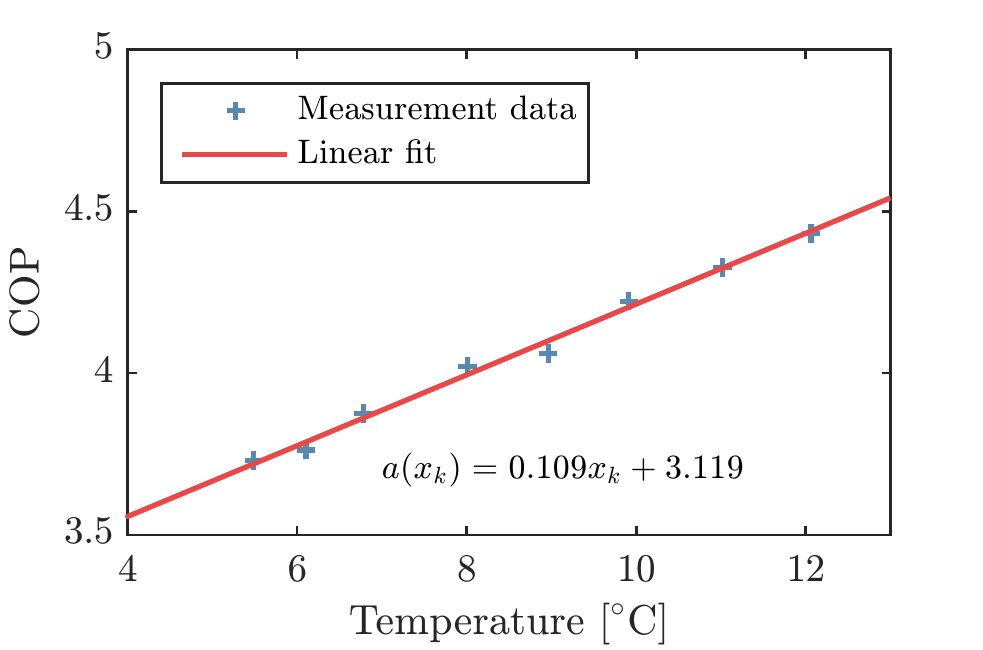}
	\caption[Fitting of the temperature-dependent COP]{Fitting of the inlet temperature-dependent \ac{COP} $a(x_t)$ of the \ac{HP} }\label{fig:cop}
\end{figure}
\begin{table}[H]
	\caption[Energy system data]{Energy system data}
	\label{table:inputdata}
	\begin{center}
		\begin{tabular}{ l l l l}
			& Parameter &Single storage&Multiple storage\\ 
			\midrule
			\textbf{Grid Feeders} \\ Power & Cost $c_e$:& 0.096\$/kWh &0.096\$/kWh\\
			\midrule
			Gas & Cost $c_g$:& 0.063\$/kWh &0.063\$/kWh\\
			\midrule
			\textbf{Conversion} \\ \ac{HP}s& \ac{COP} $a(x_t)$:& see Fig.~\ref{fig:cop}&see Fig.~\ref{fig:cop}\\& Capacity $\overline{u}_{\rm out}$:& 60kW & 60kW \\
			\midrule
			Boiler& Efficiency $a_{\textrm{b}}$:& 0.7&0.7\\ & Capacity $\overline{u}_{b}$:& 285 kW & 855 kW\\
			\midrule
			Chiller& \ac{COP} $a_{\rm ch}$:& 5&5\\
			&  Capacity $\overline{u}_{\rm ch}$: &150kW&450kW\\
			\midrule
			\textbf{Storage} \\
			Boreholes&Conductivity $\lambda$: & 0.621kW/$^\circ$C&0.621kW/$^\circ$C$\pm 10\%$\\
			& Inertia $mc$: &14805kWh/$^\circ$C&14805kWh/$^\circ$C\\
			& Capacity $\overline{u}_{in}$:&100kW&100kW\\
			& Ground $T_\infty$:&12$^\circ$C &12$^\circ$C\\
			& Range $[\underline{T},\overline{T}]$: &[0,12]$^\circ$C&[0,12]$^\circ$C\\
			\bottomrule
		\end{tabular}
	\end{center}
\end{table}
\begin{table}[H] 
	\caption[Accumulated solver time over all iterations of the Moment DPP approach]{Accumulated MOSEK\textsuperscript{TM} solver time over all iterations of the Moment DDP approach obtained on a PC with an Intel-i5 2.2GHz CPU with 8GB RAM for a tolerance of $\epsilon=10^{-4}$ (after scaling the problem data to the unit box)}\label{table:times}
	\begin{center}
		\begin{tabular}{ l l l}
			Basis functions/Relaxation&Single storage&Multiple storage\\ 
			\midrule
			Affine value functions, $2k=2$&4.77s&6.39s\\
			Affine value functions, $2k=4$&5.77s&18.65min\\
			Quadratic value functions, $2k=4$&25.23s&28.24min\\
			\midrule
		\end{tabular}
	\end{center}
\end{table}
\end{document}